 \documentclass[final,5p,times,twocolumn,sort&compress]{elsarticle}
 \biboptions{}

\usepackage{amssymb}

\usepackage[]{nomencl}

\makenomenclature

\journal{Sustainable Energy, Grids and Networks}
\usepackage{amssymb}
\usepackage{amsmath}
\usepackage{amsthm}
\usepackage{xspace}
\usepackage{cases}
\usepackage{color}
\usepackage{graphicx}
\usepackage[caption=false]{subfig}
\usepackage{epstopdf}
\usepackage{upgreek}
\usepackage{enumitem}
\usepackage{booktabs}
\usepackage{array}
\usepackage{blindtext}
\usepackage{multirow}
\usepackage{bbold}
\usepackage[hidelinks]{hyperref}
\usepackage{soul}

\newtheorem{defi}{Definition}
\newtheorem{thm}{Theorem}
\newtheorem{coro}{Corollary}
\newtheorem{propo}{Proposition}

\newtheorem{rema}{Remark}

\newtheorem{exmpl}{Example}
\newtheorem{sett}{Setting}

\newcommand{\unaryminus}{\scalebox{0.5}[1.0]{\( - \)}}

\newcommand{\nbus}{N}
\newcommand{\ngen}{N}

\newcommand{\numline}{N_l}
\newcommand{\nxi}{n_\xi}

\newcommand{\basisfun}{\psi}
\newcommand{\basisfunVector}{\bs{\basisfun}}
\newcommand{\varsqrt}[1]{S_{#1}}

\newcommand{\adjustlength}{0mm}

\newcommand{\pdf}{\textsc{PDF}\xspace}
\newcommand{\pdfs}{\textsc{PDF}s\xspace}

\newcommand{\Ltwospacern}[1]{\mathcal{L}^2 (\Omega,\mathbb{R}^{#1})}

\newcommand{\cc}{\textsc{CC}}

\newcommand{\opf}{\textsc{OPF}\xspace}

\newcommand{\ccopf}{\cc-\opf}
\newcommand{\hopf}{h\opf}
\newcommand{\dc}{\textsc{DC}\xspace}
\newcommand{\ac}{\textsc{AC}\xspace}

\newcommand{\dcopf}{\textsc{DC}-\textsc{OPF}\xspace}

\newcommand{\pce}{\textsc{PCE}\xspace}
\newcommand{\pces}{{\textsc{PCE}}s\xspace}

\newcommand{\socp}{\textsc{socp}\xspace}

\newcommand{\normaldist}{\mathsf{N}}
\newcommand{\betadist}{\mathsf{B}}

\newcommand{\std}[2]{\sigma_{#2}^{\text{#1}}}

\newcommand{\bs}[1]{\boldsymbol{#1}}
\newcommand{\rv}[1]{\mathsf{#1}}

\newcommand{\ev}[1]{\operatorname{E}\left[#1\right]}

\newcommand{\var}[1]{\operatorname{Var}\left[#1\right]}

\newcommand{\pdffun}{f}

\newcommand{\pvar}{u}
\newcommand{\pvarup}{\expandafter\uppercase\expandafter{\pvar}}

\newcommand{\pfix}{d}
\newcommand{\pfixup}{\expandafter\uppercase\expandafter{\pfix}}

\newcommand{\agc}{\textsc{AGC}\xspace}

\newcommand{\hentry}{h}

\newlist{mylist}{enumerate}{1}
\setlist[mylist]{label=\textsc{p}\oldstylenums{\arabic*}}
\newlist{mylistlist}{enumerate}{1}
\setlist[mylistlist]{label=\textsc{q}\oldstylenums{\arabic*}}

\definecolor{till}{rgb}{.0,.0,.0}
\definecolor{timm}{rgb}{.0,.0,.0} 
\definecolor{veit}{rgb}{.8,.1,.1} 
\definecolor{added}{rgb}{1,.0,.0}
\definecolor{revised}{rgb}{.2,.8,.1}

\def\till{\textcolor{till}}

\newlist{mypcelist}{enumerate}{1}
\setlist[mypcelist]{label=\textsc{p}\oldstylenums{\arabic*}}
\newlist{myinglist}{enumerate}{1}
\setlist[myinglist]{label=\textsc{i}\oldstylenums{\arabic*}}
\newlist{mysteplist}{enumerate}{1}
\setlist[mysteplist]{label=Step\xspace\oldstylenums{\arabic*}, leftmargin=35pt}

\usepackage[english]{babel}
\begin{document}

\begin{frontmatter}



\title{A Generalized Framework for Chance-constrained Optimal Power Flow}


\author{Tillmann M{\"u}hlpfordt, Timm Faulwasser, Veit Hagenmeyer}

\address{Institute for Automation and Applied Informatics, Karlsruhe Institute of Technology, Karlsruhe, Germany,\\ $\{$tillmann.muehlpfordt, timm.faulwasser, veit.hagenmeyer$\}$@kit.edu.}

\begin{abstract}
Deregulated energy markets, demand forecasting, and the continuously increasing share of renewable energy sources call---among others---for a structured consideration of uncertainties in optimal power flow problems.
The main challenge is to guarantee power balance while maintaining economic and secure operation.
In the presence of Gaussian uncertainties affine feedback policies are known to be viable options for this task.
The present paper advocates a general framework for chance-constrained \opf problems in terms of continuous random variables. It is shown that,  irrespective of the type of distribution, the random-variable minimizers 
lead to affine feedback policies.
Introducing a three-step methodology that exploits polynomial chaos expansion, the present paper provides a constructive approach to chance-constrained optimal power flow problems that does not assume a specific distribution, e.g. Gaussian, for the uncertainties.
We illustrate our findings by means of a tutorial example and a 300-bus test case.
\end{abstract}

\begin{keyword}
chance-constrained optimal power flow, uncertainties, affine policies, polynomial chaos
\end{keyword}

\end{frontmatter}


\setlength{\nomitemsep}{-\parsep}

\printnomenclature[1.5cm]
\nomenclature[1a]{$\nbus$}{Number of buses}
\nomenclature[1b]{$\mathcal{\nbus}$}{Set of bus indices}
\nomenclature[1c]{$\numline$}{Number of lines}
\nomenclature[1d]{$\mathcal{N}_l$}{Set of line indices}
\nomenclature[1e]{$\pvar$}{Controllable active power}
\nomenclature[1fa]{$\pfix$}{Uncontrollable active power}
\nomenclature[1fb]{$p_l$}{Line power flow}
\nomenclature[1g]{$\alpha$}{\agc coefficients}
\nomenclature[1h]{$J$}{Cost function}
\nomenclature[1i]{$\underline{x}$, $\overline{x}$}{Lower bound, upper bound of $x$}
\nomenclature[1j]{$\bs{1}_n$}{$n$-dimensional column vector of ones}
\nomenclature[1j]{$\phi$}{Power transfer distribution factor matrix}
\nomenclature[3a]{$\Omega$}{Set of outcomes}
\nomenclature[3b]{$\mathbb{P}$}{Probability measure}
\nomenclature[3d]{$\mathcal{L}^2(\Omega, \mathbb{R})$}{Hilbert space of second-order random variables w.r.t. probability measure $\mathbb{P}$}
\nomenclature[3e]{$\xi$}{Stochastic germ}
\nomenclature[3f]{$\basisfun_\ell$}{$\ell$\textsuperscript{th} basis function}
\nomenclature[3fa]{$\basisfunVector$}{Vectorized basis $\basisfunVector = [ \basisfun_1, \hdots, \basisfun_L ]^\top$}
\nomenclature[3g]{$\langle \cdot, \cdot \rangle$}{Scalar product}
\nomenclature[3h]{$L+1$}{\pce dimension}
\nomenclature[4a]{$\rv{x}$}{Random variable}
\nomenclature[4b]{$\tilde{x} = \rv{x}(\tilde{\xi})$}{Realization of random variable $\rv{x}$}
\nomenclature[4c]{$x_{\ell}$}{$\ell$\textsuperscript{th} vector of \pce coefficients of $\rv{x}$}
\nomenclature[4d]{$X$}{Matrix of \pce coefficients of $\rv{x}$ of degree greater zero}
\nomenclature[4f]{$\ev{\rv{x}}$}{Expected value of $\rv{x}$}
\nomenclature[4g]{$\var{\rv{x}}$}{Variance of $\rv{x}$}
\newpage
\section{Introduction}
\label{sec:Introduction}
The continuing increase in electricity generation from renewable energy sources and liberalized energy markets pose challenges to the operation of power systems \cite{Liu12}; i.e., the importance of uncertainties is on the rise.
Uncertainty leads to and/or increases fluctuating reserve capacities, and varying line power flows across the network, among others.
The structured consideration of uncertainties is thus paramount in order to ensure the economic and secure operation of power systems in the presence of fluctuating feed-ins and/or uncertain demands.

Optimal power flow (\opf) is a standard tool for operational planning and/or system analysis of power systems.
The objective is to minimize operational costs whilst respecting generation limits, line flow limits, and the power flow equations.
Assuming no uncertainties are present the solution approaches to this optimization problem are numerous, see for example references listed in \cite{Frank2012}.
In the presence of stochastic uncertainties the \opf problem must be reformulated, ensuring
\begin{enumerate}[label=(\roman*), parsep=0pt, topsep=2pt]
	\item that technical limitations (inequality constraints) are met with a specified probability, and \label{item:TechnicalLimitations}
	\item that the power flow equations (equality constraints) are satisfied for all possible realizations of the uncertainties, i.e. power system stability is achieved. \label{item:PowerFlowEquations}
\end{enumerate}
Regarding issue \ref{item:TechnicalLimitations}, chance-constrained optimal power flow (\ccopf) is a formulation that allows inequality constraint violations with the probability of constraint violation as a user-specified parameter.
Individual chance constraints admit deterministic, distributionally robust convex reformulations of the \ccopf problem \cite{Roald15b}.
For Gaussian uncertainties these reformulations are exact \cite{Bienstock14,Roald13,Roald15b}.
Alternatively, it is possible to solve the individually chance-constrained optimization problem by means of multi-dimensional integration \cite{Zhang11,Zhang13}.
\till{Scenario-based methods---often applied to multi-stage problems \cite{Vrakopoulou12,Vrakopoulou13b,Vrakopoulou2017a}---are an alternative to chance-constrained approaches; the chance constraints are replaced by sufficiently many deterministic constraints leading to large but purely deterministic  problems \cite{Campi04}.}

Regarding issue \ref{item:PowerFlowEquations}, the power flow equations are physical constraints that hold despite fluctuations.
This requires feedback control.
In particular, automatic generation control (\agc) balances mismatches between load and generation, given sufficient reserves can be activated.
Affine policies have been shown to yield power references that satisfy \dc power flow in the presence of (multivariate) Gaussian uncertainties \cite{Bienstock14,Roald13,Roald15b,Roald16a} (assuming ideal primary control).
Existing approaches \cite{Roald15b,Bienstock14,Roald13,Roald16a} to single-stage \ccopf under \dc power flow and Gaussian uncertainties directly formulate the \ccopf problem in terms of the parameters of the affine feedback, leading to finite-dimensional second-order cone programs.
However, the relevance and advantages of non-Gaussian uncertainties for modeling load patterns and renewables have been emphasized in the literature \cite{Atwa10a, Carpaneto08a, Soubdhan09a}.
Certain non-Gaussian distributions (such as Beta distributions) allow compact supports and skewed probability density functions, which hence overcome modeling shortcomings of purely Gaussian settings.
For example, to model a load via a Gaussian random variable always bears a non-zero probability for the load acting as a producer.
Arguably, this probability may be small, but an uncertainty description that rules out this possibility by design is physically consistent and desirable.

We remark that how to address the reformulation of inequality constraints in the problem formulation, i.e. issue~\ref{item:TechnicalLimitations}, is a user-specific choice.
As such, this choice resembles a trade-off between computational tractability and modeling accuracy.
In contrast, the validity of the power flow equations, i.e. issue~\ref{item:PowerFlowEquations} imposes a physical equality constraint that has to be accounted for in the problem formulation.
The present paper proposes a general framework for chance-constrained \opf that combines modeling uncertainties in terms of continuous random variables of finite variance, \emph{and} a rigorous mathematical consideration of the power flow equations as equality constraints of the \opf problem. 
It is shown that a formulation of the \ccopf problem in terms of random variables naturally leads to engineering-motivated affine policies.
Under the mild assumption that uncertainties are modeled as continuous random variables of finite variance with otherwise arbitrary probability distributions, our findings highlight that the optimal affine policies are indeed random-variable minimizers of an underlying \ccopf problem.
	
A consequence of the last item is that the proposed general framework to \ccopf embeds \emph{and} extends current approaches \cite{Bienstock14,Roald13,Roald15b,Roald16a} which consider purely Gaussian settings.

The key step is to formulate the \ccopf problem rigorously with random variables as decision variables.
This unveils the infinite-dimensional nature of \ccopf.
A three-step methodology concisely describes the proposed approach to \ccopf: formulation, parameterization, optimization.
This results in optimal affine policies that satisfy power balance despite uncertainties.
The corresponding optimization problem scales well in terms of the number of uncertainties.
For common individual chance-constraint reformulations it leads to a second-order cone program.
Polynomial chaos expansion (\pce) is employed to represent all occurring random variables by finitely many deterministic coefficients.

While \pce dates back to the late 30s \cite{Wiener38}, it has been applied to power systems only recently, for example to design a power converter \cite{monti04a}, to design observers in the presence of uncertainties \cite{smith07}, and to solve stochastic power flow \cite{Muehlpfordt16b,Muehlpfordt17a,Tang16,Ni17,Appino17}.
The applicability of \pce to \opf problems under uncertainty has been demonstrated in \cite{Muehlpfordt16b,Muehlpfordt17a,Tang16,Ni17}.
The works \cite{Tang16,Ni17} focus on computational details when implementing \pce.
In contrast, \cite{Muehlpfordt16b,Muehlpfordt17a} mention that the power flow equations are always satisfied.
However, \cite{Muehlpfordt16b,Muehlpfordt17a} do not put \pce approaches to \opf in relation to other existing approaches, and do not show optimality of affine policies.
Instead, the present paper takes a different view: starting from existing approaches \cite{Bienstock14,Roald13,Roald15b,Roald16a} we show that \pce is a generalization; the more mathematical nature of \pce is thus related to the engineering practice of affine policies.

The present manuscript focuses on a framework for single-stage \opf problems under uncertainty, highlighting the importance of affine control policies rigorously irrespective of the kind of distribution of the uncertainty.
Affine policies have also been applied to multi-stage \opf under uncertainty \cite{Vrakopoulou12,Warrington13,Vrakopoulou2017a,Li2017}, where their use is motivated based on engineering intuition.
For multi-stage \opf problems the handling of the inequality constraints is similar to single-stage \opf: it comprises analytically reformulated chance constraints \cite{Li2017}, convex reformulations \cite{DallAnese2017}, and scenario-based approaches \cite{Stai2017}.

Summing up, the contributions of our work are as follows:
We provide a problem formulation of chance-constrained \opf in terms of random variables that is shown to contain existing approaches \cite{Bienstock14,Roald13,Roald15b,Roald16a}.
We further give a rigorous proof showing when affine policies are optimal.
Additionally, we highlight an important dichotomy: optimal policies of chance-constrained \opf correspond to optimal random variables.
Finally, we provide a tractable and scalable reformulation of the random-variable problem in terms of a second-order cone program by leveraging polynomial chaos expansions.
The combination of the contributions provide a tractable framework for chance-constrained \opf.

The remainder is organized as follows:
Section \ref{sec:Random-variable-opf} introduces the \ccopf problem in terms of random variables, and demonstrates the flexibility of the proposed formulation: existing approaches for Gaussian uncertainties can be obtained as special cases (Section~\ref{sec:GaussianUncertainties}).
The observations at the end of Section~\ref{sec:Random-variable-opf} lead to a three-step methodology to \ccopf, presented in Section~\ref{sec:ThreeSteps} in greater detail.
Section~\ref{sec:PCE} introduces polynomial chaos expansion as a mathematical tool that is required to tackle Section~\ref{sec:ThreeSteps}. 
The methodology developed in Section~\ref{sec:ThreeSteps} is demonstrated for a tutorial 3-bus example in Section~\ref{sec:Example_3Bus}, \till{and a 300-bus test case in Section~\ref{sec:Example_300Bus}}.

\section{Preliminaries and Problem Formulation}
\label{sec:Random-variable-opf}
Consider a connected $\nbus$-bus electrical \till{transmission} network in steady state that is composed of linear components, \till{for which the \dc power flow assumptions are valid (lossless lines, unit voltage magnitude constraints, small angle differences).}
The $\numline$ lines have indices $\mathcal{N}_l = \{1,\hdots, \numline\}$.
For simplified presentation each bus $i \in \mathcal{\nbus} = \{1, \hdots, \nbus\}$ is assumed to be connected to one generation unit, and one fixed but uncertain power demand/generation.
The net active power realization $p \in \mathbb{R}^{\nbus}$ is $p = \pvar + \pfix$,
where $\pvar \in \mathbb{R}^{\nbus}$ represents adjustable/controllable (generated) power, and $\pfix \in \mathbb{R}^{\nbus}$ resembles (uncontrollable) power demand in case of $\pfix_i < 0$ for bus $i \in \mathcal{\nbus}$, or (uncontrollable) renewable feed-in in case of $\pfix_j > 0$ for bus $j \in \mathcal{N}$.
The goal of (deterministic) \opf is to minimize generation costs $J(u)$ with $J: \mathbb{R}^{\ngen} \rightarrow \mathbb{R}$ such that the power flow equations are satisfied (equality constraints), and generation limits and line flow limits are satisfied (inequality constraints).
Under \dc power flow conditions the standard formulation for \opf reads
\begin{subequations} 
	\label{eq:DCOPF}
	\begin{align}
	\label{eq:DCOPF_costfunction}
	\underset{\pvar \in \mathbb{R}^{\nbus}}{\operatorname{min}}~  & J(\pvar)\\
	\mathrm{s.\,t.} ~ ~ \,
	\label{eq:DCOPF_powerflow}
	& \bs{1}_{\nbus}^\top (\pvar + \pfix) = 0,\\
	\label{eq:DCOPF_cons_gen}
	& \underline{\pvar} \leq \pvar \leq \overline{\pvar}, \\
	\label{eq:DCOPF_cons_line}
	& \underline{p}_l \leq p_l =  \phi \, (\pvar + \pfix) \leq \overline{p}_l,
	\end{align}
\end{subequations}
where \eqref{eq:DCOPF_powerflow} is the power balance with $\bs{1}_{\nbus} = [1 \hdots  1]^\top \in  \mathbb{R}^{\nbus}$.
The generation limits and line limits are $\underline{\pvar}, \overline{\pvar} \in \mathbb{R}^{\ngen}$ and $\underline{p}_l, \overline{p}_l \in \mathbb{R}^{\numline}$ respectively, and $\phi \in \mathbb{R}^{\numline \times \nbus}$ is the power transfer distribution factor matrix, which maps the net power $p$ linearly to the line flows $p_l \in \mathbb{R}^{\numline}$.
\subsection{Stochastic Optimal Power Flow}
\label{sec:sOPF}
Deterministic \opf \eqref{eq:DCOPF} assumes perfect knowledge of the uncontrollable power $\pfix \in \mathbb{R}^{\nbus}$.
Instead, stochastic \opf models power consumption and/or the power feed-in due to renewables as (non-trivial) continuous second-order random vectors  $\rv{\pfix} \in \Ltwospacern{\nbus}$ with $\Omega \subseteq \mathbb{R}^{\nbus}$ as the set of possible outcomes.\footnote{More precisely, given a probability space $(\Omega, \mathcal{F}, \mathbb{P})$ the Hilbert space $\mathcal{L}^2(\Omega, \mathbb{R})$ w.r.t. measure $\mathbb{P}$ is the set of equivalence classes  modulo the almost-everywhere-equality relation of real-valued random variables $\rv{x}: (\Omega, \mathcal{F}, \mathbb{P}) \rightarrow \mathbb{R} $ with finite variance and inner product $\langle \rv{x}, \rv{y} \rangle = \ev{\rv{x}  \rv{y}}$, see~\cite{Sullivan15book}.
	We refrain from explicitly mentioning the $\sigma$-algebra $\mathcal{F}$ and the probability measure $\mathbb{P}$ when referring to the Hilbert space.
	For $\mathbb{R}^n$-valued random vectors $\rv{x}: (\Omega, \mathcal{F}, \mathbb{P}) \rightarrow \mathbb{R}^n$  we introduce the shorthand notation $\rv{x} \in \mathcal{L}^2(\Omega, \mathbb{R}^n)$ w.r.t. probability measure $\mathbb{P}$ in the sense that $\rv{x} \cong [\rv{x}_1, \hdots, \rv{x}_n]^\top$ with $\rv{x}_i \in \mathcal{L}^2(\Omega, \mathbb{R})$.
}
We will show that the probabilistic modeling of $\rv{\pfix}$ requires a reformulation of the deterministic \opf problem~\eqref{eq:DCOPF} in terms of random variables as decision variables.

\subsubsection*{Power Balance}
Power balance despite probabilistic uncertainties can be achieved by---formally---optimizing over random variables.
To see this, choose some \emph{fixed} power generation $\pvar \in \mathbb{R}^{\nbus}$ as in the \opf problem~\eqref{eq:DCOPF}.
Then, this will violate the power balance almost surely, i.e.  
\begin{equation}
\text{a.s. }\forall \tilde{\xi} \in \Omega:  ~ ~ ~ \bs{1}_{\nbus}^\top (\pvar + \tilde{\pfix}) \neq 0,
\end{equation}
where $\tilde{\pfix} = \rv{\pfix}(\tilde{\xi}) \in \mathbb{R}^{\nbus}$ is the realization of the random variable uncontrollable power $\rv{\pfix}$ for the outcome $\tilde{\xi} \in \Omega$.
Hence, we introduce a feedback policy $\rv{\pvar} = \rv{\pvar}(\rv{\pfix}) \in \Ltwospacern{\nbus}$ for the generators.
More precisely, the feedback policy $\rv{\pfix} \mapsto \rv{\pvar}(\rv{\pfix})$ maps  stochastic uncertainties~$\rv{\pfix}$ to power set points that ensure power balance despite uncertainties.
Doing so, the decision variable $\rv{\pvar}$ of \opf becomes itself a random variable.\footnote{This is also observed in \cite{Zhang11}, where it is stated that ``due to random inputs [\dots], the output variables are also random.''}
The requirement of power balance leads to the notion of viability \cite{Bienstock14}.
\begin{defi}[Viable policy \cite{Bienstock14}]
	\label{defi:viability}
	A feedback policy $\rv{\pvar} = \rv{\pvar}(\rv{\pfix})$ is called viable if for any realization of the uncertainty $\rv{\pfix}$ the power flow equations are satisfied.\footnote{This definition of viability is slightly more general than the original definition from \cite{Bienstock14}: it is not limited to affine control policies, and it is not restricted to the \dc power flow setting. In the context of power systems, viability is related to power system stability.
		Specifically for steady state power flow this amounts to power balance despite uncertain fluctuations.}
	\hfill $\square$
\end{defi}
Viability of the policy $\rv{\pvar}(\rv{\pfix})$ is ensured if the power balance is satisfied in terms of random variables
\begin{subequations}
\label{eq:RandomVariablePowerBalance}
\begin{equation}
\bs{1}_{\nbus}^\top ( \rv{\pvar}(\rv{\pfix}) + \rv{\pfix} ) = 0,
\end{equation}
or written in terms of realizations of the random variables
\begin{equation}
\forall \tilde{\xi} \in \Omega: \bs{1}_{\nbus}^\top ( \tilde{\pvar}(\tilde{\pfix}) + \tilde{\pfix}  ) = 0,
\end{equation}
where $\tilde{\pvar}(\tilde{\pfix}) = \rv{\pvar}(\rv{\pfix}(\tilde{\xi})) \in \mathbb{R}^{\nbus}$ is the control input, and $\tilde{\pfix} = \rv{\pfix}(\tilde{\xi}) \in \mathbb{R}^{\nbus}$ is the realization of the uncontrollable power.
\end{subequations}
The decision variable of stochastic \opf is thus the random vector $\rv{\pvar} \in \Ltwospacern{\nbus}$ which corresponds to a feedback control policy.

\subsubsection*{Cost Function}
For stochastic \opf the cost function has to map \emph{policies} to scalars
\begin{equation}
\label{eq:GenericCostFunction}
\hat{J}: \Ltwospacern{\nbus} \rightarrow \mathbb{R}.
\end{equation}
A typical choice is $\hat{J}(\rv{\pvar}) = \ev{J(\rv{\pvar})}$, where $\ev{\cdot}$ is the expected value \cite{Kall82}.

\subsubsection*{Inequality Constraints}
The presence of uncertainties requires the inequality constraints \eqref{eq:DCOPF_cons_gen}, \eqref{eq:DCOPF_cons_line} to be reformulated, because inequality constraints formulated in terms of random variables are in general not meaningful.
This may be done, for example, either by using chance constraints (individual/joint) or robust counterparts.
Possible chance constraint formulations are \cite{Vrakopoulou12,Roald13,Calafiore2006}
\begin{subequations}
	\label{eq:ChanceConstraintFormulations}
	\begin{align}
	\label{eq:ChanceConstraintFormulations_joint}
	\text{joint:}\: &\begin{cases}
	\mathbb{P}( (\underline{u} \leq \rv{\pvar} \leq \overline{u}) \cap (\underline{p}_l \leq \rv{p}_l \leq \overline{p}_l) ) \geq 1 - \varepsilon,
	\end{cases}\\
	\label{eq:ChanceConstraintFormulations_double}	
	\text{double:}\: &\begin{cases}
	\mathbb{P}(\underline{u}_i \leq \rv{\pvar}_i \leq \overline{u}_i) \geq 1 - \varepsilon, & i \in \mathcal{\nbus},\\
	\mathbb{P}(\underline{p}_{l,j} \leq \rv{p}_{l,j} \leq \overline{p}_{l,j}) \geq 1 - \varepsilon, & j \in \mathcal{N}_l,
	\end{cases} \\
	\label{eq:ChanceConstraintFormulations_ind}	
	\text{individual:}\:& \begin{cases}
	\mathbb{P}(\underline{u}_i \leq \rv{\pvar}_i ) \geq 1 - \varepsilon, 		&i \in \mathcal{\nbus},\\
	\mathbb{P}(\rv{\pvar}_i \leq \overline{u}_i) \geq 1 - \varepsilon,   		&i \in \mathcal{\nbus},\\
	\mathbb{P}(\rv{p}_{l,j} \leq \overline{p}_{l,j}) \geq 1 - \varepsilon,   	&j \in \mathcal{N}_l,\\		
	\mathbb{P}(\underline{p}_{l,j} \leq \rv{p}_{l,j} ) \geq 1 - \varepsilon, 	&j \in \mathcal{N}_l,\\	\end{cases}
	\end{align}
\end{subequations}
where $\rv{p}_{l,j}$ is the random-variable power flow across line $j$.

\subsubsection*{$\mathcal{L}^2$-Optimization Problem}
Using the random-variable power balance~\eqref{eq:RandomVariablePowerBalance}, the generic cost function~\eqref{eq:GenericCostFunction}, and a chance constraint formulation from~\eqref{eq:ChanceConstraintFormulations}, the chance-constrained optimal power flow (\ccopf) problem can be written as
\begin{subequations} 
	\label{eq:DCsOPF_generic}
	\begin{align}
	\underset{\rv{\pvar} \in \Ltwospacern{\ngen}}{\operatorname{min}}~  & \hat{J}(\rv{\pvar}) 
	\label{eq:DCsOPF_CostFunction}\\
	\mathrm{s.\,t.} ~ ~ \,
	& \bs{1}_{\ngen}^\top (\rv{\pvar } + \rv{\pfix}) = 0,\\
	& \text{\eqref{eq:ChanceConstraintFormulations_joint}, \eqref{eq:ChanceConstraintFormulations_double}, or \eqref{eq:ChanceConstraintFormulations_ind}}.
	\label{eq:DCsOPF_ChanceConstraintFormulation}
	\end{align}
\end{subequations}
Any feasible feedback policy of~\eqref{eq:DCsOPF_generic} is viable, see Definition~\ref{defi:viability}.
The solution of the infinite-dimensional problem~\eqref{eq:DCsOPF_generic}---assuming it exists---is the optimal viable control policy~$\rv{\pvar}^\star(\rv{\pfix})$ that yields power balance for all realizations of the uncertainty~$\rv{\pfix}$, and satisfies the constraints in a chance-constrained sense.
Notice that the random-variable power balance is a direct consequence from modeling the uncertainties $\rv{\pfix}$ as random variables.
Instead, the specific choice of the cost function~\eqref{eq:GenericCostFunction} and the chance constraint formulations are modeling degrees of freedom.
Further, notice that the \ccopf~\eqref{eq:DCsOPF_generic} does not assume a specific probability distribution for the uncertain  power~$\rv{\pfix}$.

\begin{rema}[Economic dispatch vs. \ccopf]
	\label{rema:EconomicDispatchvsCCOPF}
	\till{
	Traditionally, an economic dispatch calculation provides base points for each generation unit, given some forcast.
	Automatic generation control (\agc) with prescribed participation factors accounts for mismatches in generation and demand online, i.e. it reacts to realizations of the uncertainties \cite{Wood96book}.}
	Chance-constrained \opf attempts to unify both steps while considering the specific stochastic nature of the fluctuations, and ensuring economic and secure operation of the overall system.	
	\till{This also means that the solutions from \ccopf are applicable on the same time scale as economic dispatch and \agc.}
	\hfill $\square$
\end{rema}

\begin{rema}[Online \opf vs. \ccopf]
	\label{rema:OnlineOPFvsCCOPF}
	\till{
		Online optimal power flow is another approach to tackle \opf under uncertainty \cite{Gan2016}.
		The main idea is to measure the realization of the uncertainty, to formulate and solve a corresponding \opf problem, and to apply the generation inputs to the grid.
		The conceptual advantage of online \opf is that it acts irrespective of the distribution of the uncertainty;
		the method reacts to realizations ``online.''
		However, online \opf relies on accurate online state estimation as well as reliable, fast, repeated, and accurate solutions of large-scale \opf problems.
		Chance-constrained \opf differs conceptually as it formulates a single optimization problem that determines policies intead of set points.
		These policies can be applied in real time, similarly to \agc.
		However, chance-constrained \opf hinges on both the uncertainty model and the problem formulation.
	\hfill $\square$}
\end{rema}

\subsection{Gaussian Uncertainties}
\label{sec:GaussianUncertainties}
The purpose of this section is to show that \ccopf from~\eqref{eq:DCsOPF_generic} embeds existing approaches to \ccopf under \dc power flow and Gaussian uncertainties \cite{Roald15b,Roald16a,Roald15}.
To this end, we restate the setting from \cite{Roald15b,Roald16a,Roald15} entirely in random variables.
\begin{sett}[Linear cost, Gaussian uncertainties \cite{Roald15b,Roald16a,Roald15}]
	\label{ass:Roald}
	\mbox{}
	\begin{enumerate}
		\setlength{\itemsep}{0pt}
		\item \label{item:LinearCost} The deterministic cost function $J(\pvar) = h^\top \pvar$ is linear.
		\item \label{item:Gaussian} The uncertain power demand $\rv{\pfix}$ is modeled as
		\begin{subequations}
			\label{eq:LinesParameterization}
			\begin{equation}
			\label{eq:ForecastUncertainty}
			\rv{\pfix} = \pfix_0 +  \varsqrt\pfix \rv{\xi},
			\end{equation} 
			with $\pfix_0 \in \mathbb{R}^{\nbus}$, nonsingular $\varsqrt{\pfix} \neq 0_{\nbus \times \nbus} \in \mathbb{R}^{\nbus {\times} \nbus}$. 
			The $\nbus$-valued stochastic germ $\rv{\xi} \sim \normaldist(0,I_{\nbus})$ is a standard multivariate Gaussian.
		\item \label{item:AGCParameterization}The feedback policy $\rv{\pvar}$ is
			\begin{equation}
			\begin{aligned}
			\label{eq:AGCParameterization}
			\rv{\pvar} = \rv{\pvar}(\rv{\xi}) &= \pvar_0 - \alpha \, \bs{1}_{\ngen}^\top \varsqrt{\pfix}\rv{\xi},\\
			\bs{1}_{\nbus}^\top (\pvar_0 + \pfix_0) &= 0, \\
			1 - \bs{1}_{\nbus}^\top \alpha &= 0,
			\end{aligned}
			\end{equation}
			with unknown $\pvar_0, \alpha \in \mathbb{R}^{\ngen}$.\footnote{This choice of feedback control is also called \emph{balancing policy} \cite{Roald16a}, \emph{reserves representation} \cite{Vrakopoulou13b}, or \emph{base point/participation factor} \cite{Wood96book}.}
			\item Inequality constraints are modeled as individual chance constraints, and rewritten using the first and second moment, e.g. for upper bounds
			\begin{equation}
			\label{eq:CCReformulation}
			\ev{\rv{x}} + \beta(\varepsilon) \sqrt{\var{\rv{x}}} \leq \overline{x} \quad \Rightarrow \quad \mathbb{P}(\rv{x} \leq \overline{x}) \geq 1 - \varepsilon,
			\end{equation} 
			where $\varepsilon \in [0,1]$ is a user-defined security level.\footnote{\label{foot:MomentBasedReformulations}%
			These formulations are popular, because they are exact for Gaussian random variables (choosing $\beta$ accordingly); i.e. \eqref{eq:CCReformulation} holds in both directions.
			For general unimodal, symmetric distributions, the reformulations \eqref{eq:CCReformulation} can be conservative, nevertheless yield convex reformulations \cite{Roald15b}.
			For highly skewed distributions less conservative results can be obtained by considering higher-order (centralized) moments, i.e. skewness and/or kurtosis \cite{Popescu05}.
			}
			\hfill $\square$
		\end{subequations}
	\end{enumerate}
\end{sett}
To obtain the numerical values of the expected generation $\pvar_0$ and the coefficients $\alpha$ from Setting~\ref{ass:Roald}, \cite{Roald15,Roald15b} suggest solving the following optimization problem
\begin{subequations} 
	\label{eq:DC_OPF_Line}
	\begin{align}
	\label{eq:DCSOPF_CostFunction_Line}
	\underset{\pvar_0, \alpha \in \mathbb{R}^{\nbus}}{\operatorname{min}}~  & \hspace{1.9mm} h^\top \pvar_0\\
	\mathrm{s.\,t.} ~ ~ \,
	\label{eq:PowerBalanceLine}
	& \begin{array}{l}
	\bs{1}_{\nbus}^\top (\pvar_0 + \pfix_0) = 0,\\
	1 - \bs{1}_{\nbus}^\top \alpha = 0
	\end{array} \\
	& \label{eq:CC_factor_gen} \underline{\pvar}_i \leq \ev{\rv{\pvar}_i} \pm \beta_{\pvar} \sqrt{\var{\rv{\pvar}_i}} \leq \overline{\pvar}_i,\\
	& \label{eq:CC_factor_line}	 \underline{p}_{l,j} \leq \ev{\rv{p}_{l,j}} {\pm} \beta_{l} \sqrt{\var{\rv{p}_{l,j}}} \leq \overline{p}_{l,j}, \\	
	\nonumber	& \forall i \in \mathcal{\nbus}, ~ \forall j \in \mathcal{N}_l,
	\end{align}
	where $\rv{p}_{l,j}$ is the $j$\textsuperscript{th} entry of the line flow $\rv{p}_l = \phi \, (\rv{\pvar} + \rv{\pfix})$ with the matrix of power transfer distribution factors $\phi $.
\end{subequations}
The nominal \dc power flow and the summation condition \eqref{eq:PowerBalanceLine} ensure viability of the feedback policy $\rv{\pvar}$, cf. \cite[Lemma 2.1]{Bienstock14}.
The generation limits and line flow limits are modeled as individual chance constraints that admit the exact reformulation \eqref{eq:CC_factor_gen} and \eqref{eq:CC_factor_line} for Gaussian uncertainties \cite{Roald15b}; they correspond to the individual chance constraint formulation from \eqref{eq:ChanceConstraintFormulations}.
Let $\pvar_0^\star, \alpha^\star$ denote the optimal solution of Problem \eqref{eq:DC_OPF_Line}---assuming it exists.
Then, the optimal feedback policy is $\rv{\pvar}^\star(\rv{\xi}) = \pvar_0^\star -  \alpha^\star \bs{1}_{\ngen}^\top \varsqrt{\pfix}\rv{\xi}$.
Given the realization  $\tilde{\xi}$ of the random variable $\rv{\xi}$ the corresponding realization $\tilde{\pvar}^\star$ of the random variable $\rv{\pvar}^\star$ becomes $\tilde{\pvar}^\star = \rv{\pvar}^\star(\rv{\tilde{\xi}}) = \pvar_0^\star - \alpha^\star \bs{1}_{\ngen}^\top \varsqrt{\pfix} \tilde{\xi}$, which is the control action.

\begin{rema}[Gaussian random variables]
	\label{rema:Gaussianity}
	Under Setting \ref{ass:Roald} the power demand $\rv{\pfix}$, the feedback policy $\rv{\pvar}$, and the line power $\rv{p}_l$ are Gaussian random variables,
	\begin{subequations}
		\label{eq:GaussianVariables}
		\begin{align}
		\label{eq:DemandParameterization_Gaussian}
		\rv{\pfix} &\sim \normaldist(\pfix_0, \Sigma_{\pfix}), && \Sigma_{\pfix} = \varsqrt{\pfix} \varsqrt{\pfix}^\top \\
		\label{eq:RV_GaussianParameterization}
		\rv{\pvar} &\sim \normaldist(\pvar_0, \Sigma_{\pvar}), &&\Sigma_{\pvar} = (\alpha \bs{1}_{\nbus}^\top \varsqrt{\pfix}) (\alpha \bs{1}_{\nbus}^\top \varsqrt{\pfix})^\top,\\
		\rv{p}_l &\sim \normaldist(p_{l0}, \Sigma_{l}), && \Sigma_{l} {=} (\phi (I {-} \alpha \bs{1}^\top_{\ngen}) \varsqrt{\pfix}) (\phi (I {-} \alpha \bs{1}^\top_{\ngen}) \varsqrt{\pfix})^\top \!  ,
		\end{align}
	with all covariance matrices being positive semidefinite.
	Moreover, all random variables admit an affine parameterization
	\begin{equation}
	\rv{x} = x_0 + \Sigma_x \xi, \quad \rv{x} \in \{ \rv{\pfix}, \rv{\pvar}, \rv{p}_l\},~x \in \{d, u , l\},
	\end{equation}
	w.r.t. the stochastic germ $\rv{\xi} \sim \normaldist(0,I_{\nbus})$.
	\end{subequations}	
	\hfill $\square$
\end{rema}
Remark \ref{rema:Gaussianity} shows and emphasizes the dichotomy of $\rv{\pvar}$: it is both a random variable and a feedback policy in terms of the (realization of the) stochastic germ, i.e. once the realization of $\rv{\xi}$ is known the evaluation of $\pvar(\xi)$ yields an applicable control.\footnote{This assumes perfect and immediate measurement of the realization of $\xi$.
	State estimation of power systems is beyond the scope of this paper.}

Now we can address the question raised at the beginning of this section, namely how  \ccopf~\eqref{eq:DCsOPF_generic} relates to Problem~\eqref{eq:DC_OPF_Line}.
To this end, consider \ccopf~\eqref{eq:DCsOPF_generic} under Setting~\ref{ass:Roald} with the cost formulation $\hat{J}(\rv{\pvar}) = \ev{J(\rv{\pvar})}$.
This gives
\begin{subequations} 
	\label{eq:DC_OPF_uncertainQP}
	\begin{align}
	\label{eq:DCSOPF_CostFunction}
	\underset{\rv{\pvar} \in \Ltwospacern{\ngen}}{\operatorname{min}}~  & \mathrm{E} [{h^\top \rv{\pvar }}]\\
	\mathrm{s.\,t.} ~ ~ \,
	\label{eq:RandomVariablePowerBalanceOPF}
	& \bs{1}_{\ngen}^\top (\rv{\pvar } + \rv{\pfix}) = 0,\\
	& \nonumber \text{\eqref{eq:CC_factor_gen}, \eqref{eq:CC_factor_line}}, \, \forall i \in \mathcal{\nbus}, ~ \forall j \in \mathcal{N}_l.
	\end{align}
\end{subequations}

\begin{propo}[Equivalence of \eqref{eq:DC_OPF_Line} and \eqref{eq:DC_OPF_uncertainQP}]
	\label{propo:SpecialCase}
	Let the optimal solution to Problem~\eqref{eq:DC_OPF_Line} be $\pvar_0^\star, \alpha^\star \in \mathbb{R}^{\ngen}$.	
	Furthermore, let the optimal solution to Problem~\eqref{eq:DC_OPF_uncertainQP} be $\rv{\pvar}^\star \in \Ltwospacern{\ngen}$.
	Then,
	\begin{equation}
	\rv{\pvar}^\star = \pvar_0^\star - \alpha^\star \, \bs{1}_{\ngen}^\top \varsqrt{\pfix}\rv{\xi}
	\end{equation}
	holds such that $\rv{\pvar}^\star \sim \normaldist(\pvar_0^\star, (\alpha^\star \bs{1}_{\nbus}^\top \varsqrt{\pfix}) (\alpha^\star \bs{1}_{\nbus}^\top \varsqrt{\pfix})^\top)$.
	\hfill $\square$
\end{propo}

\begin{proof}
	First, the cost function \eqref{eq:DCSOPF_CostFunction}  becomes $\ev{h^\top 	\rv{\pvar}} = h^\top \pvar_0$, which is \eqref{eq:DCSOPF_CostFunction_Line}.
	Using the uncertainty modeling \eqref{eq:LinesParameterization} the random-variable \dc power flow \eqref{eq:RandomVariablePowerBalanceOPF} becomes
	\begin{align}
	\label{eq:PowerFlowInserted}
	\bs{1}_{\nbus}^\top (\rv{\pvar }+ \rv{\pfix}) 
	= \bs{1}_{\nbus}^\top (\pvar_0 + \pfix_0) + (1 - \bs{1}_{\nbus}^\top \alpha) \,  \bs{1}_{\nbus}^\top \varsqrt{\pfix}\rv{\xi} =: \rv{x}.
	\end{align}
	The real-valued random variable $\rv{x}$ is a linear combination of the Gaussian random variables $\rv{\pfix}$ and $\rv{\pvar}$,
	\begin{equation}
	\rv{x} \sim \normaldist(\bs{1}_{\ngen}^\top (\pvar_0 + \pfix_0), (1 - \bs{1}_{\ngen}^\top  \alpha)^2 \bs{1}_{\ngen}^\top \varsqrt{\pfix}\varsqrt{\pfix}^\top \bs{1}_{\ngen}),
	\end{equation}
	hence it is fully described by its mean and variance.
	The random variable $\rv{x}$ from \eqref{eq:PowerFlowInserted} has to degenerate to zero according to~\eqref{eq:RandomVariablePowerBalanceOPF}.
	This is the case if and only if its mean and variance are zero,
	\begin{subequations}
		\label{eq:NecessaryConditions_Line}
		\begin{align}
		\label{eq:EV_necessary}
		\ev{\rv{x}} &=  \bs{1}_{\ngen}^\top (\pvar_0 + \pfix_0) \overset{!}{=} 0\\
		\label{eq:Var_necessary}		
		\var{\rv{x}} &=  (1 - \bs{1}_{\ngen}^\top  \alpha)^2 \bs{1}_{\ngen}^\top \varsqrt{\pfix}\varsqrt{\pfix}^\top \bs{1}_{\ngen}\, \overset{!}{=} 0. 
		\end{align}
	\end{subequations}
	This means that random-variable \dc power flow under Assumption \ref{ass:Roald} is equivalent to considering the nominal \dc power flow~\eqref{eq:EV_necessary} and the summation condition \eqref{eq:Var_necessary}, cf. \eqref{eq:PowerBalanceLine}.\footnote{An alternative proof of viability is given in \cite{Bienstock14}.}
	The inequality constraints are identical by construction.
	The infinite-dimensional Problem \eqref{eq:DC_OPF_uncertainQP} is equivalently represented by a deterministic finite dimensional problem in terms of $\pvar_0$ and $\alpha$.
	As such it is equivalent to Problem \eqref{eq:DC_OPF_Line}, and the optimal feedback follows from \eqref{eq:AGCParameterization}.
\end{proof}
The consequence of Proposition \ref{propo:SpecialCase} is that Problem \eqref{eq:DC_OPF_Line} is a reformulation of the infinite-dimensional Problem \eqref{eq:DC_OPF_uncertainQP} in terms of the mean and variance of the Gaussian decision variable~$\rv{\pvar}$.
Problem \eqref{eq:DC_OPF_uncertainQP} is thus a generalization of Problem \eqref{eq:DC_OPF_Line}.
Further, the optimal policies stemming from the solution of Problem~\eqref{eq:DC_OPF_Line} are Gaussian random variables.
Gaussianity of all occurring random variables follows from linearity of the \dc power flow, cf. Remark \ref{rema:Gaussianity}.
In fact, linearity of the \dc power flow allows a more general statement when interpreted as a mapping: regardless of the distribution of the uncertainty~$\rv{\pfix}$, the control~$\rv{\pvar}$ has---qualitatively---the same distribution.
It hence stands to reason that Problem~\eqref{eq:DC_OPF_uncertainQP} allows the computation of optimal affine policies that satisfy power balance also for non-Gaussian uncertainties---which is shown in Section~\ref{sec:ThreeSteps} by means of a three-step methodology.
In preparation, however, we have to introduce polynomial chaos expansion as a tool that allows a structured treatment of random variables.

\section{Polynomial Chaos Expansion}
\label{sec:PCE}
We review the basic elements of polynomial chaos expansion in Section~\ref{sec:PCE_intro}; refer to \cite{Sullivan15book, Xiu10book} for a more detailed introduction.
\subsection{Introduction}
\label{sec:PCE_intro}
Consider $\nxi$ independent second-order random variables $\rv{\xi}_i \in \mathcal{L}^2(\Omega_i, \mathbb{R})$ w.r.t. the measure $\mathbb{P}_i$ for $i = 1,\hdots, \nxi$.
The random vector $\rv{\xi} \cong [\rv{\xi}_1,\hdots, \rv{\xi}_{\nxi}]^\top$ is called stochastic germ.
Denote by $\Omega = \Omega_1 \times \cdots \times \Omega_{\nxi}$ the support, and let $\mathbb{P} = \mathbb{P}_1 \cdots \mathbb{P}_{\nxi}$ be the product measure.
Polynomial chaos then allows to express any random variable $\rv{x} \in \mathcal{L}^2(\Omega, \mathbb{R})$ w.r.t. measure $\mathbb{P}$ as a linear combination of orthogonal $\nxi$-variate polynomials.
To this end, let the Hilbert space $\mathcal{L}^2(\Omega, \mathbb{R})$ be spanned by the set of $\nxi$-variate polynomials $\{\basisfun_\ell \}_{\ell=0}^{\infty}$ that are orthogonal, i.e.
\begin{equation}
\label{eq:orthogonality}
\ev{\basisfun_\ell \basisfun_k} = \langle \basisfun_\ell, \basisfun_k \rangle = \int \basisfun_\ell(\tau) \basisfun_k(\tau)  \mathrm{d}\mathbb{P}(\tau) = \gamma_\ell \delta_{\ell k},
\end{equation}
for all $\ell,\, k \in \mathbb{N}_0$, where $\gamma_\ell$ is a positive constant, and $\delta_{\ell k}$ is the Kronecker-delta.\footnote{The (non-unique) multivariate basis can be constructed from the univariate bases cf. \cite{Sullivan15book}.}
W.l.o.g. we choose $\basisfun_0 = 1$.\footnote{For univariate stochastic germs one chooses the basis s.t. $\operatorname{deg} \basisfun_\ell = \ell$.}
Polynomial chaos allows to rewrite $\mathbb{R}^{n_x}$-valued random vectors $\rv{x} = [\rv{x}_1, \hdots, \rv{x}_{n_x}]^\top$ with elements $\rv{x}_i \in \mathcal{L}^2(\Omega, \mathbb{R})$ for $i = 1, \hdots, n_x$ as
\begin{subequations}
	\begin{align}
	\label{eq:PCE_infty}
	\rv{x} &=  	\sum_{\ell=0}^\infty x_\ell \basisfun_\ell \text{ with } x_\ell = [x_{1,\ell}, \hdots, x_{n_x,\ell}]^\top \in \mathbb{R}^{n_x},
	\end{align}
	and
	\begin{equation}
	\label{eq:PCEcoeffs}
	x_{i,\ell} = \frac{\langle \rv{x}_i, \basisfun_\ell \rangle}{\langle \basisfun_\ell, \basisfun_\ell \rangle} \in \mathbb{R}.
	\end{equation}		
	The vector $x_{\ell}$ contains all so-called \pce coefficients  \eqref{eq:PCEcoeffs}.
\end{subequations}
For numerical implementations the infinite sum \eqref{eq:PCE_infty} is truncated after $L + 1 \in \mathbb{N}$ terms
\begin{subequations}
	\label{eq:PCE_complete}	
	\begin{align}
	\rv{x} \approx \hat{\rv{x}} &= \sum_{\ell=0}^{L} x_\ell \basisfun_\ell = x_0 + \uppercase{x} \basisfunVector, \\
	\text{with ~} \uppercase{x} &= [x_1\quad \hdots \quad x_L] \in \mathbb{R}^{n_x \times L},\\
	\basisfunVector &= [\basisfun_1 \quad \hdots \quad \basisfun_L]^\top.
	\end{align}
\end{subequations}
The truncation error $\|\rv{x}_i - \hat{\rv{x}}_i \| $ is $\mathcal{L}^2$-optimal, i.e. $\lim_{L\to \infty} \| \rv{x}_i - \hat{\rv{x}}_i \| = 0$, in the induced norm $\| \cdot \|$, see \cite{Sullivan15book, Xiu10book}. 
The dimension of the subspace $\operatorname{span} \, \{\basisfun_\ell\}_{\ell=0}^L \subseteq \mathcal{L}^2(\Omega, \mathbb{R})$ is
\begin{equation}
\label{eq:PCEdimension}
L + 1=\frac{(\nxi+n_d)!}{\nxi ! n_d!},
\end{equation}
where $\nxi$ is the dimension of the stochastic germ and $n_d$ the highest polynomial degree.
Hence, the subspace is spanned by orthogonal basis polynomials in $n_\xi$ variables of degree at most $n_d$.
Statistics of $\rv{x}$ can be obtained directly from the \pce coefficients, without having to sample, e.g. \cite{Sullivan15book,Xiu10book}
\begin{equation}
\label{eq:PCE_statistics}
\ev{\hat{\rv{x}}} = x_0, ~ ~ \var{\hat{\rv{x}}} = \sum_{\ell=1}^{L} x_\ell x_\ell^\top \langle \basisfun_\ell, \basisfun_\ell \rangle.
\end{equation}
\begin{rema}[Higher-order moments]
	\label{rema:HigherOrderMoments}
	The fact that moments of the random variable $\hat{\rv{x}}$ can be computed from its \pce coefficients alone can be leveraged for moment-based bounds or bounds that consider more than two moments \cite{Popescu05}.
	That is, highly skewed distributions may require more than two moments to reduce conservatism in \eqref{eq:CC_factor_gen} and \eqref{eq:CC_factor_line}.
	For which moments to include see, e.g. \cite{Popescu05}.
	If additional moments are included \pce is still applicable, but comes at a higher computational cost potentially, since convexity may be lost.
	\hfill $\square$
\end{rema}
The applicability of \pce hinges on the dimension $(L + 1)$ from~\eqref{eq:PCEdimension}.
It is desirable to have both a low dimension \emph{and} exact expansions.
How to achieve this is considered next.

\subsection{Exact Affine Polynomial Chaos}
\label{sec:ExactAffinePCE}
We are interested in random variables that admit a truncated polynomial chaos expansion that is both \emph{exact} in the sense that~\eqref{eq:PCE_complete} holds with equality, and \emph{affine} in the sense that the maximum degree of the polynomial basis is one.
\begin{defi}[Exact affine \pce]
	\label{defi:ExactAffinePolynomialChaos}
	A random variable $\rv{x} \in \mathcal{L}^2(\Omega, \mathbb{R})$ is said to have an exact affine \pce if
	\begin{equation}
	\label{eq:ExactAffine_Condition}
	\rv{x} = \tilde{\rv{x}} = \sum_{\ell=0}^{L} x_\ell \basisfun_\ell \quad \text{with} \quad \underset{\ell = 0, \hdots, L}{\operatorname{max}}\, (\operatorname{deg} \, \basisfun_\ell) = 1,
	\end{equation}
	i.e., the orthogonal polynomial basis $\{ \basisfun_\ell \}_{\ell = 0}^L$ has at most degree one.
	The $\mathbb{R}^n$-valued random vector $x \in \mathcal{L}^2(\Omega,\mathbb{R}^n)$ is said to have an exact affine \pce if every entry $\rv{x}_i$ admits an exact affine \pce \eqref{eq:ExactAffine_Condition} for $i = 1, \hdots, n$.
	\hfill $\square$
\end{defi}
In other words---Definition~\ref{defi:ExactAffinePolynomialChaos} demands the polynomial basis $\{\basisfun_\ell\}_{\ell=0}^L$ to be at most affine, i.e. the stochastic germ $\xi$ appears in no higher order.
It remains to discuss how to choose this specific polynomial basis.
For a variety of well-known univariate distributions, for example Beta, Gamma, Gaussian, or Uniform distributions, the orthogonal bases that yield exact affine \pces~\eqref{eq:PCE_complete} are known.
In the remainder we refer to these uncertainties as \emph{canonical uncertainties}.
\begin{coro}[Exact affine \pce for canonical uncertainties \cite{Xiu10book}]
	\label{rema:ExactAffinePCE}
	Let $\rv{x} \in \mathcal{L}^2(\Omega, \mathbb{R})$ follow one of the following univariate distributions: Beta, Gamma, Gaussian, or Uniform.
	Then, for the following stochastic germs $\xi$
	\begin{center}
		\small
		\begin{tabular}{l   l  l  l  l}
			Distribution & Support & Polyn. Basis &  $\basisfun_\ell$ & Notation\\
			\hline
			Beta & $[0,1]$ & Jacobi & $P_\ell^{(\alpha,\beta)}$ & $\xi \sim \betadist(\alpha, \beta)$\\
			Gamma  & $[0,\infty)$ & Gen. Laguerre & $L_\ell$ & $\xi \sim \Gamma(p)$\\
			Gaussian  & $(-\infty,\infty)$ & Hermite & $H\!e_\ell$ & $\xi \sim \normaldist(0,1)$\\
			Uniform  & $[0,1]$ & Legendre & $P_\ell$ & $\xi \sim \mathsf{U}(0,1)$
		\end{tabular}
	\end{center}
	\normalsize
	the \pce of $\rv{x}$
	\begin{equation*}
	\rv{x} = \hat{\rv{x}} = \sum_{\ell = 0}^{1} x_\ell \basisfun_\ell = x_0 + x_1 \basisfun_1
	\end{equation*}
	is an exact affine \pce in the basis $\{ \basisfun_\ell \}_{\ell = 0}^1$.
	\hfill $\square$
\end{coro}
\begin{exmpl}
	\label{exmpl:UnivariateGaussian}
	Any univariate Gaussian random variable $\rv{x}~\,\sim\,~\normaldist(\mu,\sigma^2)$ admits the exact affine \pce with respect to the stochastic germ $\xi \sim \normaldist(0,1)$
	\begin{equation*}
	\rv{x} = \mu + \sigma \xi = x_0 + x_1 H\!e_1,
	\end{equation*}
	where $[x_0, x_1]  ^\top = [\mu, \sigma]^\top$ are the \pce coefficients, and $\{\basisfun_\ell\}_{\ell=0}^1 = \{H\!e_\ell\}_{\ell=0}^1 = \{1, \xi\}$ is the affine Hermite polynomial basis that is orthogonal w.r.t. the univariate Gaussian measure, namely
	\begin{align*}
	\langle \basisfun_0, \basisfun_0 \rangle& = \langle H\!e_0, H\!e_0 \rangle = \frac{1}{\sqrt{2 \pi}}\,\int_{\mathbb{R}} 1 \, \mathrm{e}^{-\frac{\tau^2}{2}} \mathrm{d}\tau = 1, \\
	\langle \basisfun_1, \basisfun_1 \rangle &= \langle H\!e_1, H\!e_1 \rangle = \frac{1}{\sqrt{2 \pi}}\,\int_{\mathbb{R}} \tau^2\, \mathrm{e}^{-\frac{\tau^2}{2}} \mathrm{d}\tau = 1, \\
	\langle \basisfun_0, \basisfun_1 \rangle &= \langle H\!e_0, H\!e_1 \rangle = \frac{1}{\sqrt{2 \pi}}\,\int_{\mathbb{R}} \tau\, \mathrm{e}^{-\frac{\tau^2}{2}} \mathrm{d}\tau = 0.
	\end{align*}
	\till{As can be seen, the univariate Gaussian admits an exact affine \pce in the Hermite polynomial basis.
	Of course it is possible to derive a polynomial chaos expansion of a Gaussian random variable in a different orthogonal basis, for example in a Legendre basis.
	In that case, however, the \pce of $\rv{x}$ will not be exact and affine with respect to this Legendre basis.}
\hfill $\square$
\end{exmpl}
Going beyond canonical uncertainties, the exact affine \pce for any random variable of finite variance is obtained as follows.
\begin{propo}[Exact affine \pce for non-canonical uncertainty]
	\label{propo:BeyondCanonicalUncertainty}
	Let $\rv{x} \in \mathcal{L}^2(\Omega, \mathbb{R})$ with probability measure $\mathbb{P}$ be given.
	Then, $\rv{x}$ can be used directly as the stochastic germ, and the \pce of $\rv{x}$
	\begin{equation}
	\label{eq:PCE_NonCanonical}
	\rv{x} = \hat{\rv{x}} = \sum_{\ell=0}^{1} x_\ell \basisfun_\ell = x_0 + x_1 \basisfun_1
	\end{equation}
	is exact affine with \pce coefficients $[x_0, x_1]^\top = [ \ev{\rv{x}}, 1 ]^\top$ w.r.t. the orthogonal basis $ \{ \basisfun_\ell \}_{\ell = 0}^1 = \{ 1, \rv{x} - \ev{\rv{x}} \}$.
	\hfill $\square$
\end{propo}
\begin{proof}
	Direct inspection of~\eqref{eq:PCE_NonCanonical} with the given \pce coefficients shows that $\rv{x} = \rv{x}$.
	Orthogonality of the basis $\{ \basisfun_\ell \}_{\ell = 0}^1$ is shown by verifying the orthogonality condition~\eqref{eq:AGCParameterization},	
		\begin{align*}
		\langle \basisfun_0, \basisfun_0 \rangle 	&= \int_{} 1 \, \mathrm{d}\mathbb{P}(\tau) = 1 \\
		\langle \basisfun_0, \basisfun_1 \rangle 	&= \int_{} 1\,(\tau - \ev{\rv{x}_i}) \mathrm{d}\mathbb{P}(\tau)
		= \int_{} \, \tau \mathrm{d}\mathbb{P}(\tau) - \ev{\rv{x}} = 0 \\
		\langle \basisfun_1, \basisfun_1 \rangle 	&= \int_{} \, (\tau - \ev{\rv{x}_i})^2 \mathrm{d}\mathbb{P}(\tau) = \var{\rv{x}},
		\end{align*}
	which completes the proof.
\end{proof}
\begin{exmpl}
Consider a continuous random variable $\rv{x}$ with probability density $f_{\rv{x}}: [0,1] \rightarrow \mathbb{R}_{\geq}$, and $f_{\rv{x}}(x) = \pi/2 \, \sin( \pi x )$.
According to Proposition~\ref{propo:BeyondCanonicalUncertainty}, an orthogonal basis is then
\begin{subequations}
	\begin{align*}
	\basisfun_0 &= 1, \\
	\basisfun_1 &= \rv{x} - \ev{\rv{x}} = \rv{x} - \frac{\pi}{2}\,\int_0^1 \! \! \tau \, \sin (\pi \tau) \, \mathrm{d}\tau = \rv{x} - \frac{1}{2},
	\end{align*}
	with $\langle \basisfun_1, \basisfun_1 \rangle = \var{\rv{x}} = 1/4 - 2/\pi^2$.
	The \pce coefficients of $\rv{x}$ become $[x_0, x_1]^\top = [1/2, 1]^\top$.
	\hfill $\square$
\end{subequations}
\end{exmpl}

Both Corollary~\ref{rema:ExactAffinePCE} and Proposition~\ref{propo:BeyondCanonicalUncertainty} consider a single univariate uncertainty.
This may not suffice when modeling uncertainties for \ccopf problems, where multiple sources of uncertainty are present (for example solar, wind, demand).
However, the combination of several independent exact affine \pces can still be cast as a multivariate exact affine \pce.
\begin{propo}[Multivariate exact affine \pce]
	\label{propo:ExactAffinePCE_combination}
	Consider $n$ independent random variables with exact affine \pces $\rv{x}_i \in \mathcal{L}^2(\Omega_i, \mathbb{R})$ with respective \pces $\rv{x}_i = x_0^i + x_1^i \basisfun_1^i$ for $i = 1, \hdots, n$.
	Then, the Hilbert space $\mathcal{L}^2(\Omega_1 \times \cdots \times \Omega_n, \mathbb{R})$ w.r.t. probability measure $\mathbb{P} = \mathbb{P}_1 \cdots \mathbb{P}_n$ is spanned by the orthogonal basis
	\begin{equation}
	\{\basisfun_\ell \}_{\ell = 0}^L = \{1, \basisfun_1^1, \basisfun_1^2, \hdots, \basisfun_1^n  \}
	\end{equation}
	of dimension $L + 1 = n + 1$, cf. \eqref{eq:PCEdimension} with $\nxi = n$, and $n_d = 1$.
	Let $e_i$ be the $i$\textsuperscript{th} unit vector of $\mathbb{R}^n$, and let $\basisfun = [\basisfun_1, \hdots, \basisfun
	_n]^\top$.
	Then, the \pce
	\begin{equation}
	\rv{x}_i \cong x_0^i + e_i^\top \basisfun
	\end{equation}
	recovers the $i$\textsuperscript{th} random variable $\rv{x}_i$.
	\hfill $\square$
\end{propo}
\begin{proof}
	The assertion follows from the fact that the space over the product probability space $\mathcal{L}^2(\Omega_1 \times \cdots \times \Omega_n, \mathbb{R})$ is isomorphic to the Hilbert space tensor product of the Hilbert spaces  $\mathcal{L}^2(\Omega_i, \mathbb{R})$, see \cite{Sullivan15book}.
\end{proof}

\begin{exmpl}
	\label{exmpl:MultiGaussian}
	We illustrate how \pce applies to the multivariate Gaussian random variables from Setting \ref{ass:Roald}:
	The stochastic germ $\rv{\xi} \sim \normaldist(0,I_{\nbus})$ is $\mathbb{R}^{\nbus}$-valued and standard Gaussian.
	According to Corollary~\ref{rema:ExactAffinePCE} every univariate $\xi_i$ requires a Hermite polynomial basis $\{ H\!e_\ell^i \}_{\ell=0}^1 =  \{1, \xi_i \}$ of degree at most $n_d=1$  for all $i = 1, \hdots, \nxi$ with $\nxi=\nbus$.
	The tensorized \pce-basis from Proposition~\ref{propo:ExactAffinePCE_combination} is the $\nbus$-variate Hermite polynomial basis $\{ \basisfun_\ell \}_{\ell = 0}^{L} =  \{1, \xi_1, \hdots, \xi_{\nbus} \}$ of dimension $L + 1 = \nbus + 1$.
	Orthogonality holds w.r.t. to the multivariate Gaussian measure, i.e. for all $i,j = 0, \hdots, N$
	\begin{align*}
	\label{eq:MultivariateGaussianOrthogonality}
	\langle \basisfun_i, \basisfun_j \rangle = \left(\sqrt{2 \pi}\right)^{-N} \! \! \int_{}  \basisfun_i(\tau) \basisfun_j(\tau) \, \mathrm{e}^{- \frac{\tau^\top \tau}{2}} \mathrm{d} \tau = 
	\begin{cases}
	0, \, i \neq j\\
	1, \, i = j.
	\end{cases}
	\end{align*}
	The \pce \eqref{eq:PCE_complete} for the uncertain power demand \eqref{eq:ForecastUncertainty} from Setting \ref{ass:Roald} is then exact and affine
	\begin{equation*}
	\label{eq:PCEparameterization_fixed}
	\rv{\pfix} = \hat{\rv{\pfix}} =  d_0 + \varsqrt{\pfix} \xi = \sum_{\ell = 0}^{L} \pfix_\ell \basisfun_\ell =  d_0 + \pfixup \basisfunVector
	\end{equation*}
	with $\basisfun = \xi = [\xi_1 ~ \hdots ~ \xi_{\nbus}]^\top$.
	The $\mathbb{R}^{\nbus}$-valued \pce coefficients $\pfix_\ell$ for $\ell = 1, \hdots, \nbus$ correspond to the columns of $\varsqrt{\pfix}$.
	Similarly, the \pce for the feedback policy \eqref{eq:AGCParameterization} from Setting~\ref{ass:Roald} is exact and affine
	\begin{equation*}
	\label{eq:PCEparameterization_gen}
	\rv{\pvar} =  \pvar_0 -\alpha \bs{1}_{\nbus}^\top \varsqrt{\pfix} \xi = \sum_{\ell = 0}^{L} \pvar_\ell \basisfun_\ell
	= \pvar_0 + \pvarup \basisfunVector,
	\end{equation*}
	where the real-valued \pce coefficients $\pvar_\ell$ for $\ell = 1, \hdots, \nbus$ correspond to the columns of $-\alpha \bs{1}_{\nbus}^\top \varsqrt{\pfix}$.
	\hfill $\square$
\end{exmpl}

\section{Stochastic \opf in Three Steps}
\label{sec:ThreeSteps}
Recall that Section~\ref{sec:Random-variable-opf} concluded that \ccopf according to Problem~\eqref{eq:DC_OPF_uncertainQP} provides optimal affine feedback policies irrespective of the probability distribution of~$\rv{\pfix}$.
On the other hand, Example~\ref{exmpl:MultiGaussian} from Section~\ref{sec:PCE} analyzes Setting~\ref{ass:Roald} for \ccopf from the point of view of \pce.
The result is that the affine feedback from Setting \ref{ass:Roald} is equivalent to a \pce, and the solution to Problem~\eqref{eq:DC_OPF_Line} provides the optimal \pce coefficients from \eqref{eq:PCEparameterization_gen}.
In the following, both observations will be connected and it will be shown that the \ccopf Problem~\eqref{eq:DC_OPF_uncertainQP} can be reformulated as an optimization problem in terms of \pce coefficents of optimal feedback policies.
More precisely, the following three-step methodology is proposed:
\begin{mysteplist}
	\item \label{item:Step_Formulation} Formulate a random-variable \opf problem.
	\item 	\label{item:Step_Parameterization} Introduce an affine, viable feedback.
	\item 	\label{item:Step_Optimization} Determine the values of the feedback parameters by means of a suitable optimization problem.
\end{mysteplist}
We apply the three-step methodology exemplarily and demonstrate the role of \pce in it, \till{more specifically the role of exact affine \pce as introduced in Section~\ref{sec:ExactAffinePCE}}.
That is, from this point on we associate with \ref{item:Step_Formulation} the specific \ccopf problem
\begin{subequations} 
	\label{eq:DCsOPF_individual}
	\begin{align}
	\underset{\rv{\pvar} \in \Ltwospacern{\ngen}}{\operatorname{min}}~  & \mathrm{E}[J(\rv{\pvar})] 
	\label{eq:DCsOPF_CostFunction}\\
	\mathrm{s.\,t.} ~ ~ \,
	\label{eq:DCsOPF_PowerBalance}	
	& \bs{1}_{\ngen}^\top (\rv{\pvar } + \rv{\pfix}) = 0,\\
	& \text{\eqref{eq:ChanceConstraintFormulations_ind}}.
	\label{eq:DCsOPF_ChanceConstraintFormulation}
	\end{align}
\end{subequations}
Problem~\eqref{eq:DCsOPF_individual} is studied for the following setting:
\begin{sett}[Quadr. cost, exact affine \pce \cite{Muehlpfordt17a}]
	\label{ass:stochdemandPCE}
	\mbox{}
	\begin{enumerate}
		\setlength{\itemsep}{0pt}
		\item The deterministic cost $J(\pvar) = \pvar^\top H \pvar + h^\top \pvar$ is quadratic and positive (semi-)definite in $\pvar$.
		\item The uncertainty $\rv{\pfix}$ admits an exact affine \pce
		\begin{subequations}
			\begin{equation}
			\label{eq:ExactPCE_fixedpower}
			\rv{\pfix} = \sum_{\ell=0}^{L} \pfix_\ell \basisfun_\ell = \pfix_0 + \pfixup \basisfunVector,
			\end{equation}
			and the \pce coefficients $\pfix_\ell \in \mathbb{R}^{\nbus}$ are known w.r.t. the polynomial basis $\{\basisfun_\ell\}_{\ell=0}^{L}$.
			\item The feedback policy $\rv{\pvar}$ is given by its finite \pce 
			\begin{equation}
			\begin{aligned}
			\label{eq:ExactPCE_gen}
			\rv{\pvar}  &= \sum_{\ell = 0}^{L} \pvar_{\ell} \basisfun_\ell = \pvar_0 + \pvarup \basisfunVector,\\
			\bs{1}_{\nbus}^\top (\pvar_\ell + \pfix_\ell) &= 0, \quad \ell = 0, 1, \hdots, L,
			\end{aligned}
			\end{equation}
		\end{subequations}
		where the \pce coefficients $\pvar_{\ell} \in \mathbb{R}^{\ngen}$ are decision variables 
		w.r.t. the polynomial basis $\{\basisfun_\ell\}_{\ell=0}^{L}$.
		\item The individual chance constraints \eqref{eq:ChanceConstraintFormulations_ind} are rewritten using the first and second moment, cf. Setting~\ref{ass:Roald}.
		\hfill $\square$
	\end{enumerate}
\end{sett}
The main assumption in Setting~\ref{ass:stochdemandPCE} is finiteness and exactness of the \pce for the uncertainty $\rv{\pfix}$---which is no severe restriction, cf. Section~\ref{sec:PCE}.\footnote{Note that Setting \ref{ass:stochdemandPCE} contains Setting \ref{ass:Roald}: quadratic costs generalize linear costs; Gaussian random variables admit an exact affine \pce, cf. \eqref{eq:PCEparameterization_fixed} and \eqref{eq:PCEparameterization_gen}; the feedback parameterization is affine and viable; and the inequality constraint modeling is equivalent.}
Setting~\ref{ass:stochdemandPCE} also addresses  \ref{item:Step_Parameterization}: it contains an affine feedback parameterization.
It remains to show that this feedback parameterization is indeed viable.
\begin{thm}[Viable feedback control via \pce]
	\label{propo:PCEgeneralizesAGC}
	Let the uncertain power demand $\rv{\pfix}$ admit an exact affine \pce \eqref{eq:ExactPCE_fixedpower}.
	Then, any viable feedback control \eqref{eq:ExactPCE_gen} admits an exact affine \pce.
	\hfill $\square$
\end{thm}
\begin{proof}
	Viability of the feedback control \eqref{eq:ExactPCE_gen} follows from linearity of the random-variable \dc power flow \eqref{eq:RandomVariablePowerBalanceOPF}.
	Under Setting \ref{ass:stochdemandPCE} Equation \eqref{eq:RandomVariablePowerBalanceOPF} becomes
	\begin{equation}
	\label{eq:MinimalParameterization}
	\bs{1}_{\nbus}^\top ( \rv{\pfix}  + \rv{\pvar}) = \sum_{\ell=0}^{L} \bs{1}_{\nbus}^\top \pfix_\ell \basisfun_\ell  + \bs{1}_{\nbus}^\top \rv{\pvar} = 0.
	\end{equation}
	The basis $\{ \basisfun_\ell \}_{\ell=0}^{L}$---which is at most affine by assumption---spans the $(L{+}1)$-dimensional subspace $\mathcal{M} =  \operatorname{span} \{ \basisfun_\ell \}_{\ell=0}^{L} \subseteq \mathcal{L}^2(\Omega, \mathbb{R})$.\footnote{In principle any basis $\mathcal{B}$ spanning $\mathcal{M}$ is possible.
		The choice of the basis $\{\basisfun_\ell\}_{\ell=0}^{L}$ is natural, because orthogonality is computationally advantageous as it allows the deterministic reformulation \eqref{eq:DetReformulation} of the random-variable power flow equations.
		It further yields $\mathcal{L}^2$-optimality.
		Furthermore, applying Gram-Schmidt orthogonalization to $\mathcal{B}$ leads to the basis $\{\basisfun_\ell\}_{\ell=0}^{L}$.}
	All entries $\rv{\pfix}_i$ of $\rv{\pfix}$ are elements of~$\mathcal{M}$.
	To attain a feasible feedback control for \eqref{eq:MinimalParameterization} it is necessary and sufficient to choose $\rv{\pvar}_i \in \mathcal{M}$; hence $\rv{\pvar}$ is exact affine.
	To attain zero in \eqref{eq:MinimalParameterization}, which corresponds to viability of the feedback control, $\sum_{\ell = 0}^{L} \bs{1}_{\nbus}^\top (\pfix_\ell + \pvar_{\ell}) \, \basisfun_\ell = 0$ must hold.
	Project onto the non-zero basis functions $\basisfun_\ell$, and exploit orthogonality
	\begin{equation}
	\label{eq:DetReformulation}
	\bs{1}_{\nbus}^\top (\pfix_\ell + \pvar_{\ell}) \,\langle  \basisfun_\ell,  \basisfun_\ell \rangle = 0, \quad \ell = 0, 1, \hdots, L,
	\end{equation}
	which holds if and only if $\bs{1}_{\nbus}^\top (\pfix_\ell + \pvar_{\ell}) = 0$, because $\langle  \basisfun_\ell,  \basisfun_\ell \rangle = \| \basisfun_\ell \|^2 $ is positive.
\end{proof}
In short, Theorem~\ref{propo:PCEgeneralizesAGC} shows that affine policies are viable in case of exact affine \pces for the uncertainty $\rv{\pfix}$.
Importantly, viability of the affine feedback is a consequence only of linearity of the random-variable \dc power flow.

Finally, \ref{item:Step_Optimization} needs to  be addressed, i.e. how to compute the \pce coefficients $\pvar_\ell$ tractably and efficiently.
To this end, properties of \pce introduced in Section~\ref{sec:PCE} are exploited to reformulate the cost function, the equality constraints, and the inequality constraints from Problem~\eqref{eq:DCsOPF_individual} in terms of \pce coefficients.

\subsubsection*{Cost Function}
Under Setting~\ref{ass:stochdemandPCE} the cost function from Problem~\eqref{eq:DCsOPF_individual} reads
$
\mathrm{E}[J(\rv{\pvar})] = \mathrm{E} [ \rv{\pvar}^\top H \rv{\pvar} + h^\top \rv{\pvar} ]
$ for positive definite~$H$.
Exploiting orthogonality of the \pce basis, the cost function becomes \cite{Muehlpfordt17a}
\begin{subequations}
\label{eq:Ingredients_Rewrite_PCEDCOPF}
\begin{align}
\label{eq:EVcost_PCE}
\ev{J(\rv{\pvar})} & = J(\ev{\rv{\pvar}}) + \sum_{\ell = 1}^{L} \gamma_\ell \, \pvar_\ell^\top H \pvar_\ell,
\end{align}
where the positive constant $\gamma_\ell$ is computed offline according to~\eqref{eq:orthogonality}.
The cost \eqref{eq:EVcost_PCE} is equal to the sum of the cost of the expected value $J(\ev{\rv{\pvar}}) = J(\pvar_0)$ and the cost of uncertainty.
The cost of uncertainty has two origins: the monetary quadratic costs $H$, and the genuine cost of uncertainty due to uncertainty encoded by higher-order coefficients $\pvar_\ell$ for $\ell = 1, \hdots, L$.

\subsubsection*{Equality Constraints} 
Power balance~\eqref{eq:DCsOPF_PowerBalance} in terms of random variables holds iff the power balance holds in terms of the \pce coefficients,
\begin{equation}
\bs{1}_N^\top (\pfix_\ell + \pvar_\ell) = 0, \quad \ell = 0, \hdots, L.
\end{equation}
This has been proved in Theorem~\ref{propo:PCEgeneralizesAGC}, and is contained in \eqref{eq:ExactPCE_gen} from Setting~\ref{ass:stochdemandPCE}.

\subsubsection*{Inequality Constraints}
Moments of a random variable can be rewritten in terms of \pce coefficients, cf. \eqref{eq:PCE_statistics}.
For the control input $\rv{\pvar}$ this leads to
\begin{align}
\mathrm{E}[\rv{\pvar}_i] + \beta_{\pvar} \sqrt{\mathrm{Var} [\rv{\pvar}_i]} = 
\rv{u}_{i,0} + \beta_{\pvar} \sqrt{ \sum\nolimits_{\ell = 1}^L \gamma_\ell u_{i,\ell}^2 }\,,
\end{align}
\end{subequations}
where $u_{i,\ell}$ is the $\ell$-th \pce coefficient of the control input at bus~$i$.
The same procedure applies to the uncertain line flows $\rv{p}_{l,j}$ for lines $l \in \mathcal{N}_l$.

\subsubsection*{Optimization Problem}
Combining the results~\eqref{eq:Ingredients_Rewrite_PCEDCOPF}, \ref{item:Step_Optimization} leads to the following convex program \cite{Muehlpfordt17a}

\begin{subequations}
	\label{eq:DC_OPF_SOCP}
	\begin{align}
	\underset{\pvar_{\ell} \in \mathbb{R}^{\nbus} \forall \ell = 0, \hdots, L}{\operatorname{min}}\quad  & J(\pvar_0) +  \sum_{\ell = 1}^{L}  \gamma_\ell \, \pvar_\ell^\top H \pvar_\ell\\
	\label{eq:DCSOPF_EnergyBalance_PCE}
	\mathrm{s.\,t.} ~ ~ \,
	& \bs{1}_N^\top (\pfix_\ell + \pvar_\ell) = 0, \quad \ell = 0, \hdots, L, \\
	& \underline{\pvar}_i \leq \pvar_{i,0} \pm \beta_{\pvar} \, \sqrt{\sum\nolimits_{\ell=1}^L \gamma_\ell \, \pvar_{i,\ell}^2} \leq \overline{\pvar}_i, \\
	& \underline{p}_{l,j} \leq  p_{l,j,0} \pm \beta_{l} \, \sqrt{ \sum\nolimits_{\ell=1}^L \gamma_\ell \, p_{l,j,\ell}^2} \leq \overline{p}_{l,j}, \\
	\nonumber & \forall i \in \mathcal{\nbus}, ~ \forall j \in \mathcal{N}_l.
	\end{align}
\end{subequations}
Problem~\eqref{eq:DC_OPF_SOCP} is a convex second-order cone program (\socp) where the \pce coefficients of the control inputs are decision variables.
The optimal policy is obtained from the optimal solution $u_{\ell}^\star$ of \socp~\eqref{eq:DC_OPF_SOCP} as follows
\begin{equation}
\label{eq:RVSolution}
\rv{\pvar}^\star  = \sum_{\ell=0}^{L} \pvar_\ell^\star \basisfun_\ell = \pvar_0^\star + \pvarup^\star \basisfun \in  \Ltwospacern{\ngen}.
\end{equation}
Given a particular realization $\tilde{\pfix}$ of the uncertain demand, there exists a corresponding realization $\tilde{\xi}$ of the stochastic germ, i.e. $\tilde{\pfix} = \rv{\pfix}(\tilde{\xi})$.
This, in turn, results in a realization of the feedback policy~\eqref{eq:RVSolution}, namely the control input
\begin{equation}
\label{eq:RVRealization}
\tilde{\pvar}^\star = \rv{\pvar}^\star(\tilde{\xi}) = \sum_{\ell=0}^{L} \pvar_\ell^\star \basisfun_\ell(\tilde{\xi}) = \pvar_0^\star + \pvarup^\star \basisfun(\tilde{\xi}) \in \mathbb{R}^{\ngen}.
\end{equation}
The control action \eqref{eq:RVRealization} satisfies the \dc power flow equations by construction.
Similarly, the probability for violations of generation limits and/or line limits is accounted for.
\begin{rema}[Realization $\tilde{\pfix}$]
	In practice it is the $N$-valued realization $\tilde{\pfix}$ of the power demand that is accessible, from which the $\nxi$-valued realization $\tilde{\xi}$ of the stochastic germ has to be computed.
	This is straightforward for multivariate exact affine uncertainties for which $\nxi = L$, see Proposition~\ref{propo:ExactAffinePCE_combination}.
	All basis polynomials of degree one can be concisely written as an $\mathbb{R}^{L}$-valued affine function ${\basisfunVector} = [\basisfun_1, \hdots,  \basisfun_L]^\top =
	a + B \xi$ with $a \in \mathbb{R}^L$ and nonsingular $B \in \mathbb{R}^{L \times L}$.
	Rearranging \eqref{eq:ExactPCE_fixedpower} leads to
	\begin{equation}
	\begin{aligned}
	\label{eq:LGS}
	\pfixup B \tilde{\xi} &= \tilde{\pfix} - \pfix_0 - \pfixup a.
	\end{aligned}
	\end{equation}
	A mild assumption is $\nbus \geq \nxi = L$, i.e. there are as many or more buses in the network than there are modeled sources of uncertainties.
	If the matrix $DB$ has full column rank, then the system of linear equations \eqref{eq:LGS} admits a unique solution by construction, because the realization $\tilde{\xi}$ must be in the range of the rectangular matrix $\pfixup B \in \mathbb{R}^{\nbus \times L}$.
	\hfill $\square$
\end{rema}
\subsubsection*{Numerical Scalability}
\label{sec:NumericalScalability}
The \socp \eqref{eq:DC_OPF_SOCP} is a tractable convex reformulation of the \ccopf Problem~\eqref{eq:DCsOPF_individual}, and it exhibits structural equivalence: the cost function remains quadratic and positive definite, the random-variable \dc power flow remains linear. Also, the reformulated chance constraints are second-order cone constraints.
Compared to a standard deterministic \dcopf Problem~\eqref{eq:DCOPF} with $\ngen$ decision variables $\pvar_i$ for $i \in \mathcal{\nbus}$, the \socp~\eqref{eq:DC_OPF_SOCP} has $\ngen (L+1)$ decision variables $\pvar_{i,\ell}$ for $i \in \mathcal{\nbus}$ and $\ell = 0, \hdots, L$; for every bus $i$ the $(L+1)$ \pce coefficients have to be computed.
In order to solve Problem~\eqref{eq:DC_OPF_SOCP}, the positive numbers $\gamma_\ell$ have to be computed.
This can be done offline, for example via Gauss quadrature \cite{Sullivan15book}.

In general, the \pce dimension $(L+1)$ grows rapidly with the number of sources of uncertainty and the required univariate basis dimensions, see~\eqref{eq:PCEdimension}.
However, problems with many sources of uncertainty are intrinsically complex, hence are expected to be challenging both conceptually and numerically.
If exact affine uncertainties are used to model uncertainties, then the required dimension of the univariate bases reduces significantly.
For example, if every bus is modeled by a distinct exact affine uncertainty, then the number of decision variables of the \socp \eqref{eq:DC_OPF_SOCP} becomes $N(L+1) = N(N+1)$.

\begin{figure}
	\centering
	\includegraphics[]{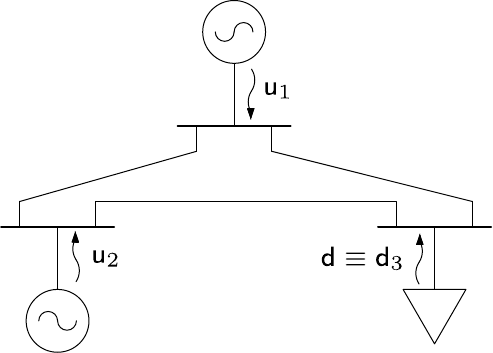}
	\caption{Three-bus system with two generators and one uncertain load.}
	\label{fig:ThreeBus}
	\vspace{\adjustlength}
\end{figure}

\begin{rema}[Local balancing via \pce]
	\pce naturally yields local balancing policies \cite{Roald16a}, because every bus can react to all uncertainties individually, hence especially to its local source of uncertainty.
	Local imbalances can be accounted for locally, thus avoiding network congestions.
	Instead, the \agc feedback from Setting \ref{ass:Roald} is a global balancing policy: the \agc coefficient $\alpha_j$ reacts to the sum of all fluctuations.
	This is motivated by current practice, where the sum of all fluctuations may be the only available broadcast signal by the \textsc{tso}.
	Then, however, local imbalances are, potentially, accounted for globally, thus possibly creating network congestions.
	Importantly, the \pce coefficients can be constrained to also yield a global balancing policy by enforcing equality of all non-zero-order \pce coefficients.
	\hfill $\square$
\end{rema}

\section{Case Studies}
\till{%
This section demonstrates the proposed method to \ccopf for two test systems.
Section~\ref{sec:Example_3Bus} considers a 3-bus example in a tutorial style: the system is simple enough to provide an analytical solution, hence allows for insightful comparisons; canonical and non-canonical uncertainties are considered.
Section~\ref{sec:Example_300Bus} provides a concise study of the 300-bus test case.
All units are given in per-unit values.}
\subsection{Tutorial 3-Bus Example}
\label{sec:Example_3Bus}
Consider the grid depicted in Figure \ref{fig:ThreeBus}: a connected 3-bus network where buses~1 and~2 each have generators but zero power demand, and bus~3 has no generator but non-zero stochastic power demand.
With slight abuse of notation set $\rv{\pvar} = [\rv{\pvar}_1, \rv{\pvar}_2]^\top \in \Ltwospacern{2}$ and $\rv{\pfix} \equiv \rv{\pfix}_3 \in \Ltwospacern{}$.
\subsubsection{Beta Distribution}
\label{sec:Beta_distribution}
\begin{figure}
	\centering
	\subfloat[Beta distribution.\label{fig:CaseBeta_demand}]{\includegraphics{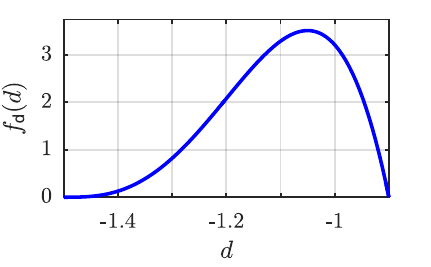}}
	\subfloat[Sinusoidal distribution.\label{fig:CaseSinusoidal_demand}]{\includegraphics{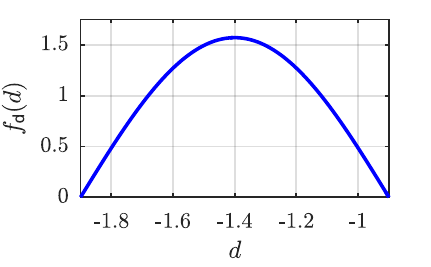}}	
	\caption{Probability density of uncertain demand $\rv{\pfix}$.}
	\vspace{\adjustlength}
\end{figure}
We consider \ccopf under Setting~\ref{ass:stochdemandPCE} with the following data:
The separable quadratic cost has parameters $H = \operatorname{diag}(0.2, 0.2)$, $h = [\hentry_1, \hentry_2]^\top = [0.5,0.6]^\top$.
The uncertain power demand~$\rv{\pfix}$ is modeled as a Beta distribution with support $[\unaryminus 1.5, \unaryminus 0.9]$ and shape parameters $a=4$, $b=2$.
According to Corollary~\ref{rema:ExactAffinePCE} it admits an exact affine \pce $\rv{\pfix} = \sum_{\ell = 0}^1 \pfix_\ell \basisfun_\ell$ w.r.t. the stochastic germ $\xi \sim \betadist(a, b)$ in a Jacobi polynomial basis $\{ \basisfun_\ell \}_{\ell = 0}^{1}  = \{ P_{\ell}^{(b \unaryminus 1,a \unaryminus 1)}(2\xi-1) \}_{\ell = 0}^{1} = \{ 1, 6\xi -4 \}$ with \pce coefficients $[\pfix_0, \pfix_1]^\top = [ \unaryminus 1.1, 0.1 ]^\top$,
\begin{equation}
\label{eq:Example_Germ2Demand}
\rv{\pfix} = \pfix_0 + \pfix_1 \basisfun_1 = \unaryminus 1.5 + 0.6 \xi \quad \Longleftrightarrow \quad \xi = \frac{\rv{\pfix} + 1.5}{0.6}.
\end{equation}
Figure \ref{fig:CaseBeta_demand} shows the skewed probability density function (\pdf) $\pdffun_{\rv{\pfix}}: [\unaryminus 1.5, \unaryminus 0.9] \rightarrow \mathbb{R}_{\geq 0}$ of the demand $\rv{\pfix}$.
Furthermore, an upper generation limit $\overline{\pvar}_1 = 0.85 $ for the generator at bus~1 is considered via the chance constraint reformulation 
\begin{equation}
\label{eq:CCFormulation_Example}
\ev{\rv{\pvar}_1} + \beta_{\pvar} \sqrt{\var{\rv{\pvar}_1}} \leq \overline{\pvar}_1 \quad \Rightarrow \quad \mathbb{P}(\rv{\pvar}_1 \leq \overline{\pvar}_1) \geq 1 - \varepsilon.
\end{equation}
The chance constraint parameter is $\beta_u = \sqrt{(1-\varepsilon)/\varepsilon}$ for a risk level $\varepsilon \in  \{ 0.05, 0.10 \}$, which is a distributionally robust formulation \cite{Calafiore2006}.
From Theorem~\ref{propo:PCEgeneralizesAGC} it is known that any viable feedback policy has to be exact affine w.r.t. the Jacobi polynomial basis, i.e. $\rv{\pvar} = \sum_{\ell = 0}^1 \pvar_{\ell} \basisfun_\ell$.
The \socp~\eqref{eq:DC_OPF_SOCP} for this setup becomes
\begin{subequations}
	\label{eq:DC_sOPFExample}
	\begin{align}
	\underset{\pvar_0, \pvar_1 \in \mathbb{R}^{2}}{\operatorname{min}}\quad  & {\frac{1}{2}\, \pvar_0^\top  H \pvar_0 + h^\top \! \pvar_0} + \frac{\gamma_1}{2}\,   \pvar_1^\top H \pvar_1\\
	\mathrm{s.\,t.} ~ ~ \,
	& \pfix_\ell + \bs{1}_2^\top \pvar_\ell = 0, \quad \ell = 0, 1 \\
	& \pvar_{1,0} + \beta_{\pvar} \, \sqrt{ \gamma_1 \, \pvar_{1,1}^2} \leq \overline{\pvar}_1, 
	\end{align}
\end{subequations}
with $\gamma_1 = \langle \basisfun_1, \basisfun_1 \rangle = \frac{2}{35 \, B(4,2)}$, where $B(\cdot,\cdot)$ is the Beta function.
Having solved~\eqref{eq:DC_sOPFExample}, the optimal \pce becomes $\rv{\pvar}^\star = \pvar_0^\star + \pvar_1^\star \basisfun_1$ in terms of the stochastic germ~$\xi$; see Table~\ref{tab:CaseNew} for numerical values of the optimal \pce coefficients.
For practical considerations the optimal feedback policy in terms of the uncertain demand~$\rv{\pfix}$ is of interest.
Using~\eqref{eq:Example_Germ2Demand}, this yields the following policies
\begin{align}
	\label{eq:Case_beta_policies}
\rv{\pvar}^\star(\rv{\pfix})
= \begin{bmatrix}
\rv{\pvar}_1(\rv{\pfix}) \\
\rv{\pvar}_2(\rv{\pfix})
\end{bmatrix}
=
\begin{cases}
\,
\begin{bmatrix}
\phantom{\unaryminus}0.6513 \\
\unaryminus 0.6513
\end{bmatrix}
-
\begin{bmatrix}
0.1270\\
0.8730
\end{bmatrix}
\,
\rv{\pfix},
& \varepsilon = 0.05, \\
\,
\begin{bmatrix}
\phantom{\unaryminus} 0.58 \\
\unaryminus 0.58
\end{bmatrix}
-
\begin{bmatrix}
0.19 \\
0.81
\end{bmatrix}
\,
\rv{\pfix},
& \varepsilon = 0.10,
\end{cases}
\end{align}
which satisfy the power balance, e.g. for $\varepsilon = 0.05$
\begin{equation}
\rv{\pfix} + \bs{1}_2^\top \rv{\pvar}^\star = \rv{\pfix} + \bs{1}_2^\top \left(
\begin{bmatrix}
\phantom{\unaryminus}0.6513 \\
\unaryminus 0.6513
\end{bmatrix} - 
\begin{bmatrix}
0.127 \\ 0.873
\end{bmatrix}\, \rv{\pfix}
\right)
= 0.
\end{equation}
\begin{figure*}
	\centering
	\subfloat[Optimal policies.\label{fig:Beta_policies}]{\includegraphics{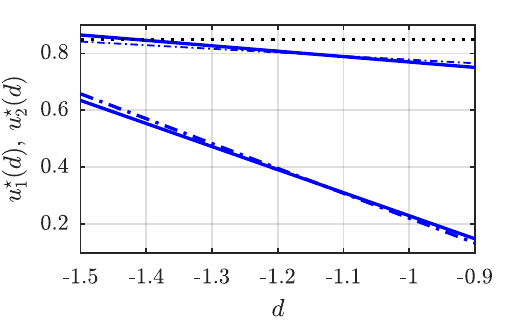}}	%
	\subfloat[Probability density of generator 1.\label{fig:Beta_PDF_1}]{\includegraphics{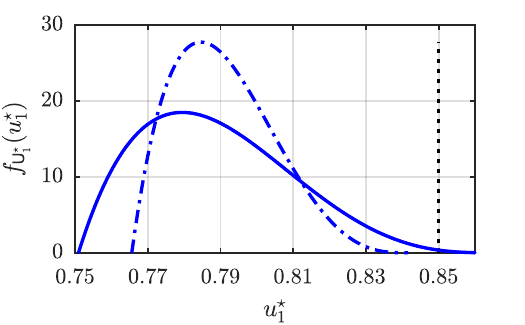}}	%
	\subfloat[Probability density of generator 1.\label{fig:Beta_PDF_2}]{\includegraphics{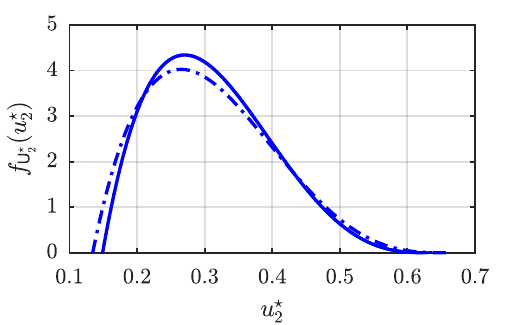}}	%
	\caption{Demand as Beta distribution---Results for security levels $\varepsilon = 0.05$ (dash-dotted), and  $\varepsilon = 0.10$ (solid). Upper generation limit shown dotted.\label{fig:NewCase}}
	\subfloat[Optimal policies.\label{fig:Sin_policies}]{\includegraphics{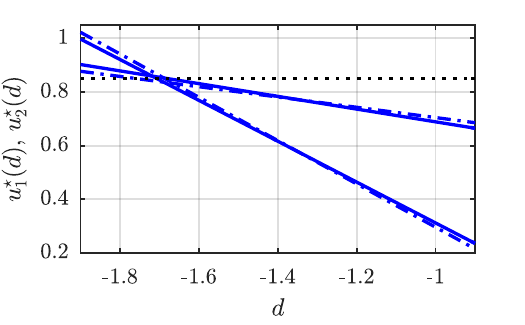}}	%
	\subfloat[Probability density of generator 1.\label{fig:Sin_PDF_1}]{\includegraphics{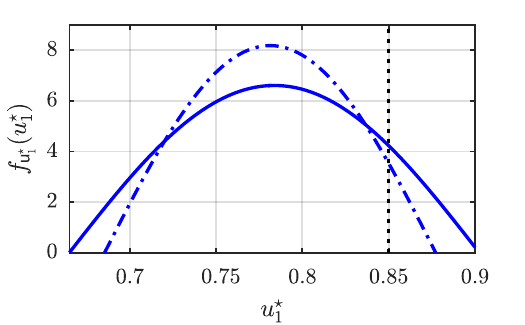}}	%
	\subfloat[Probability density of generator 1.\label{fig:Sin_PDF_2}]{\includegraphics{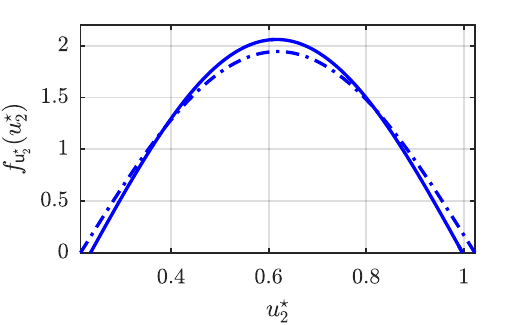}}	%
	\caption{Demand as sinusoidal distribution---Results for security levels $\varepsilon = 0.05$ (dash-dotted), and  $\varepsilon = 0.10$ (solid). Upper generation limit shown dotted.\label{fig:NewCase}}
	\vspace{\adjustlength}		
\end{figure*}

\begin{table}
	\centering
	\caption{Optimal \pce coefficients and policies for $\varepsilon \in \{ 0.05, 0.10 \}$.
		\label{tab:CaseNew}}
	\small
	\begin{tabular}{ccccc}
		\toprule
		& \multicolumn{2}{c}{Section~\ref{sec:Beta_distribution}---Beta} & \multicolumn{2}{c}{Section~\ref{sec:Sinusoidal_distribution}---Sinusoidal}
		\\
		$\varepsilon$ & $\pvar_0^\star$ 	& 	$\pvar_1^\star$ & $\pvar_0^\star$ 	& 	$\pvar_1^\star$  \\
		\midrule 						
		0.05 & $\begin{bmatrix}
		0.7910 \\ 0.3090
		\end{bmatrix}$ 	&
		$\unaryminus
		\begin{bmatrix}
		0.0127 \\ 0.0873
		\end{bmatrix}
		$
		&
		$\begin{bmatrix}
		0.7813 \\ 0.6187
		\end{bmatrix}$
		&
		$\unaryminus
		\begin{bmatrix}
		0.1919 \\ 0.8081
		\end{bmatrix}
		$
		\\
		0.10
		&
		$\begin{bmatrix}
		0.7890 \\ 0.3110
		\end{bmatrix}$ 	&
		$\unaryminus
		\begin{bmatrix}
		0.0190 \\ 0.0810
		\end{bmatrix}
		$
		&
		$\begin{bmatrix}
		0.7837 \\ 0.6163
		\end{bmatrix}$ 	&
		$\unaryminus
		\begin{bmatrix}
		0.2376\\ 0.7624
		\end{bmatrix}
		$
		\\
		\bottomrule
	\end{tabular}
	\normalsize
\end{table}

The policies~\eqref{eq:Case_beta_policies} are plotted in Figure~\ref{fig:Beta_policies}: generator~1 takes a bigger share in power generation because it is cheaper.
A reduced risk level makes the slope of the policy of generator 1 more horizontal, hence more constraint-averse.
The policies map a specific realization of the demand to a specific realization of the generation; this does not contain any information about \emph{how often} certain control actions are taken.
To obtain this information about frequency of occurrences the \pdfs of the optimal feedback policies have to  be studied; they are plotted in Figures~\ref{fig:Beta_PDF_1} and \ref{fig:Beta_PDF_2}.
They show \emph{how often} certain optimal inputs are applied.
One observes that the chance constraint remains (numerically) non-binding for either choice of $\varepsilon$.
Analytically, the constraint violation probability becomes
\begin{equation}
\label{eq:ClosedFormCDF}
\mathbb{P}(\pvar_1^\star \leq \overline{\pvar}_1) = 1 - F^{(\alpha, \beta)}\left( \frac{1}{6} \left(  \frac{\overline{\pvar}_1^\star - \pvar_{1,0}^\star }{\pvar_{1,1}^\star} + 4 \right) \right),
\end{equation}
where $\pvar_{1,0}^\star$, $\pvar_{1,1}^\star$ are the zero- and first-order \pce coefficient of generator~1, and $F^{(a, b)}$ is the cumulative distribution function of the Beta distribution with shape parameters $a$, $b$.
Inserting the numerical values from Table~\ref{tab:CaseNew} yields
\begin{subequations}
	\begin{align}
	\varepsilon = 0.05: \quad &\mathbb{P}(\rv{\pvar}_1^\star \leq 0.85) = 1 \geq 1 - 0.05, \\
	\varepsilon = 0.10: \quad &\mathbb{P}(\rv{\pvar}_1^\star \leq 0.85) = 1 \geq 1 - 0.10.
	\end{align}
\end{subequations}
This shows that the employed chance constraint reformulation~\eqref{eq:CCFormulation_Example} is conservative for the considered skewed distribution from Figure~\ref{fig:CaseBeta_demand}.
To get the least conservative results, the closed-form~\eqref{eq:ClosedFormCDF} should be used directly in the optimization~\eqref{eq:DC_sOPFExample}, possibly leading to nonconvexities, hence harder problems to solve.
	
\subsubsection{Sinusoidal Distribution}
\label{sec:Sinusoidal_distribution}
Sticking to the aforementioned grid from Figure~\ref{fig:ThreeBus} let us consider \ccopf under Setting~\ref{ass:stochdemandPCE} with the following data:
The separable quadratic cost has parameters $H = \operatorname{diag}(0.2,0.1)$, $h = [0.5, 0.6]^\top$.
The uncertain power demand $\rv{\pfix}$ follows a sinusoidal distribution on $[\unaryminus1.9, \unaryminus 0.9]$ with \pdf $\pdffun_{\rv{\pfix}}(d) = \pi/2\, \sin(\pi (\pfix + 1.9) )$, see Figure~\ref{fig:CaseSinusoidal_demand}.
Choosing the stochastic germ according to the example from Section~\ref{sec:ExactAffinePCE}, $\rv{\pfix}$ admits an exact affine \pce w.r.t. the orthogonal basis $\{ 1, \xi - 1/2 \}$ with \pce coefficients $[\pfix_0, \pfix_1]^\top = [\unaryminus 1.4, 1.0]^\top$, i.e.
\begin{equation}
\label{eq:Case_Sinusoidal_Germ}
\rv{\pfix} = \pfix_0 + \pfix_1 \basisfun_1 = \unaryminus 1.4 + (\xi - 0.5) \quad \Longrightarrow \xi = \rv{\pfix} + 1.9.
\end{equation}
The upper generation limit remains $\overline{\pvar}_1 = 0.85$, as does the chance constraint reformulation~\eqref{eq:CCFormulation_Example}.
However, the chance constraint parameter is chosen as $\beta_u = \Phi^{-1}(1-\varepsilon)$ for risk levels $\varepsilon \in \{ 0.05, 0.10 \}$, where $\Phi^{-1}$ is the inverse of the cumulative distribution function of a standard Gaussian random variable.
The chance constraint reformulation is exact for Gaussian random variables \cite{Calafiore2006}.
It is chosen here because the \pdf of $\rv{\pvar}$ is symmetric and unimodal.
Solving \eqref{eq:DC_sOPFExample} yields the optimal \pce coefficients, listed in Table~\ref{tab:CaseNew}.
The optimal policies are, using \eqref{eq:Case_Sinusoidal_Germ},
	\begin{align}
		\label{eq:Case_sinusoidal_policies}
	\rv{\pvar}^\star(\rv{\pfix})
	= \begin{bmatrix}
	\rv{\pvar}_1(\rv{\pfix}) \\
	\rv{\pvar}_2(\rv{\pfix})
	\end{bmatrix}
	=
	\begin{cases}
	\,
	\begin{bmatrix}
	\phantom{\unaryminus}0.5126 \\
	\unaryminus 0.5126
	\end{bmatrix}
	-
	\begin{bmatrix}
	0.1919 \\
	0.8081
	\end{bmatrix}
	\,
	\rv{\pfix},
	& \varepsilon = 0.05, \\
	\,
	\begin{bmatrix}
	\phantom{\unaryminus} 0.4511 \\
	\unaryminus 0.4511
	\end{bmatrix}
	-
	\begin{bmatrix}
	0.2376 \\
	0.7624
	\end{bmatrix}
	\,
	\rv{\pfix},
	& \varepsilon = 0.10,
	\end{cases}
	\end{align}
which satisfy the power balance, e.g. for $\varepsilon = 0.05$
\begin{equation}
\rv{\pfix} + \bs{1}_2^\top \rv{\pvar}^\star = \rv{\pfix} + \bs{1}_2^\top \left( 
\begin{bmatrix}
\phantom{\unaryminus} 0.5126 \\ \unaryminus 0.5126
\end{bmatrix}
-
\begin{bmatrix}
0.1919 \\ 0.8081
\end{bmatrix}
\rv{\pfix}
  \right)
  = 0.
\end{equation}
The policies~\eqref{eq:Case_sinusoidal_policies} are plotted in Figure~\ref{fig:Sin_policies}:
While the risk level does not have a big influence on the policies, the upper constraint generation $\overline{\pvar}_1$ is clearly violated for generator~1 for high demands ($< \unaryminus 1.7$). 
The \pdfs of the generators are plotted in Figures~\ref{fig:Sin_PDF_1} and \ref{fig:Sin_PDF_2}.
The \pdf of generator~1 confirms that the constraint violation occurs with non-negligible probability.
The empirical constraint violation is given by
\begin{equation}
\mathbb{P}( \rv{\pvar}_1^\star \leq \overline{\pvar}_1) = \frac{1}{2} \left(1 + \cos \left( \pi \left(  \frac{\overline{\pvar}_1 - \pvar_{1,0}^\star}{\pvar_{1,1}^\star} + \frac{1}{2} \right)\right) \right),
\end{equation}
where $\pvar_{1,0}^\star$, $\pvar_{1,1}^\star$ are the zero- and first-order \pce coefficient of generator~1.
Inserting the numerical values from Table~\ref{tab:CaseNew} yields
\begin{subequations}
	\begin{align}
	\varepsilon = 0.05: \quad &\mathbb{P}(\rv{\pvar}_1^\star \leq 0.85) = 0.9511 \geq 1 - 0.05, \\
	\varepsilon = 0.10: \quad &\mathbb{P}(\rv{\pvar}_1^\star \leq 0.85) = 0.9105 \geq 1 - 0.10.
	\end{align}
\end{subequations}
Hence, the employed chance constraint formulation is adequate for this case, because the optimization problem~\eqref{eq:DC_sOPFExample} remains convex, yet it is not overly conservative.

\subsection{300-Bus Example}
\label{sec:Example_300Bus}
\begin{figure*}[ht!]
	\centering
	\includegraphics[]{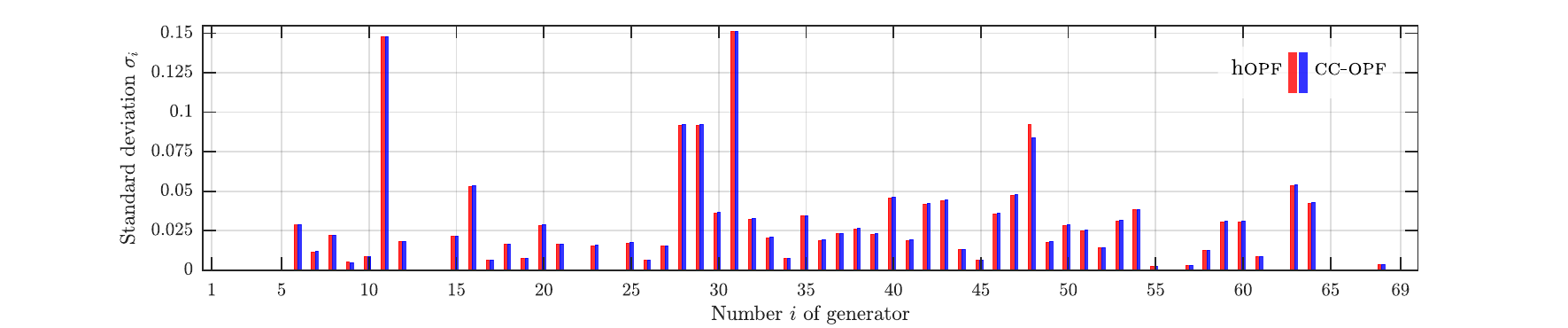}	
	\vspace{-3mm}
	\caption{Standard deviations $\sigma_i$ for generation buses computed by in-hindsight \opf (\hopf) and chance-constrained \opf (\ccopf). \label{fig:300Bus_STD}}
\end{figure*}
This subsection considers the \textsc{ieee} 300-bus test case which has $N = 300$ buses, $N_g = 69$ generation units, and $N_l = 411$ lines.
The results build on the simulation study from~\cite{Muehlpfordt17a}, but the present work considers twice as many sources of uncertainties (20 in the present paper vs. 10 in \cite{Muehlpfordt17a}).

Specifically, we introduce  18 sources of uncertainties for power the demand at buses
\begin{equation}
\begin{split}
\mathcal{N}_u &= \{ 14,    17,    61,    77,   120,   121,   122,   139,   192,   204,   \\ & \qquad \quad 217,   218,   225,   228,   231,   234,   235,   246  \} \subseteq \mathcal{N},
\end{split}
\end{equation}
where buses $i \in \{ 14,    \hdots,   218 \} \subset \mathcal{N}_u$ are modeled as Gaussians, and buses $i \in \{ 225, \hdots, 246 \} \subset \mathcal{N}_u$ are modeled as Beta distributions (Table~\ref{tab:BetaDistribution} lists the shape parameters).
The mean and variance for the Gaussian uncertainties are computed such that the $\pm 3 \sigma$-interval contains $\pm 15\,\%$ deviations from the nominal value.
Similarly, the support of the Beta distributions is the interval of $\pm 15\,\%$ deviations from the nominal value.

Additionally, there is one source of uncertainty for wind power (modeled as a Gaussian), and one source of uncertainty for solar power (modeled as a Beta distribution with shape parameters (7,7)), yielding a total of $\nxi = 20$ sources of uncertainties.
Solar and wind power are modeled to affect all buses.
The constraint reformulation parameters $\beta_{\pvar}$ and $\beta_{l}$ are chosen as $\beta_{\pvar} = \beta_{l} = \sqrt{(1 - \varepsilon)/\varepsilon}$ with violation probability~$\varepsilon = 2.5\, \%$.
This choice ensures distributionally robust satisfaction of the chance constraints~\cite{Calafiore2006}.

The optimal \pce coefficients are the solution to the~\socp~\eqref{eq:DC_OPF_SOCP}, from which optimal policies are recovered.\footnote{\till{The \socp is solved in 300\,ms with Matlab and Gurobi.}}
To assess the quality of the optimized policies via \ccopf we compare them to the policies from the most-informative case, namely in-hindsight \opf (\hopf) \cite{Muehlpfordt18b}; in-hindsight policies are obtained by sampling the uncertainties, and then solving a deterministic \opf problem for every sample.\footnote{\till{We selected 20\,000 samples. Solving the 20\,000 deterministic \opf problems took about 12.5\,min (without any parallelization).}}
The policy from \hopf provides the best distribution of optimal inputs and satisfies the constraints strictly for every sample.

For the chance constraint reformulations to be exact, it is important that \pce computes the moments accurately.
Let us assess the quality of the standard deviation---which is related to the second moment---as it acts as an uncertainty margin for the chance constraint reformulations, see~\eqref{eq:DC_OPF_SOCP}.
Figure~\ref{fig:300Bus_STD} shows two values for the standard deviations, for all generation buses; one is computed by \hopf, the other is computed by \ccopf.
As can be seen from Figure~\ref{fig:300Bus_STD} the numerical values are close.
In terms of the 1-norm of the vectors of standard deviations, the power fluctuation from \ccopf is marginally bigger, namely $\|\std{\ccopf}{} \|_{1} = 1.7035 = \| \std{\hopf}{}\|_1  + 0.0030$.
A closer look at Figure~\ref{fig:300Bus_STD} reveals that the standard deviation from \ccopf appears to be slightly larger for all buses except for bus~48, where it is significantly lower, $\| \std{\hopf}{} - \std{\ccopf}{} \|_{\infty}  = | \std{\hopf}{48} - \std{\ccopf}{48} | = 0.0090$.
The generator at bus~48 is connected to line~394 for which the line flow limit becomes binding.
Compared to the \hopf solution, \ccopf is more conservative and reduces the variability of the power injections at bus~48.

To conclude, \ccopf finds by means of a single optimization problem optimal policies that satisfy the power balance strictly and that satisfy the inequalities in a chance constraint sense.
The statistics of the policies from \ccopf are consistent with the results from the best-informed \hopf solution, yet obtained at lower computation times.

\begin{table}
	\centering
	\caption{Shape parameters of Beta distributions for 300-bus example.\label{tab:BetaDistribution}}
	\small
	\begin{tabular}{p{1.5cm}cccccc}
		\toprule
		Bus &	225 & 228 & 231 & 234 & 235 & 246 \\
		Parameters & (8,3) & (3,8) & (8,3) & (3,8) & (7,7) & (3,8)\\
		\bottomrule
	\end{tabular}
\end{table}

\section{Conclusions and Outlook}
\label{sec:Conclusions}
Given the continuously increasing share of renewables, the structured consideration of uncertainties for optimal power flow problems is paramount.
Existing approaches to chance-constrained \opf under \dc conditions often employ affine feedback policies to account for fluctuations modeled via multivariate Gaussian uncertainties.
Starting from a chance-constrained optimal power flow problem written in terms of random variables, this paper shows that the importance of affine policies is not related to the uncertainty model, e.g. Gaussian, but to the power balance constraint that maps random variables to random variables.
In addition, the optimal affine policies are random variables whose realizations satisfy the power balance, and satisfy the inequality constraints in the chance constraint sense.
The present paper proposes a three-step methodology to solving \ccopf problems, namely formulation, parameterization, optimization.
When formulating \ccopf problems, it is argued that the choice of the cost function and chance constraint formulation are modeling choices, hence user-specific.
In any case, an affine parameterization of the optimal policy is required due to power balance.
The resulting parameterized optimization problem can be solved tractably and efficiently using polynomial chaos expansion, which results for example in a second-order cone program that scales well with the number of uncertainties.
The proposed methodology is applied to a tutorial 3-bus example both for a standard and non-standard uncertainty model, and to the \textsc{ieee} 300-bus system.

The present paper has not investigated the relation between different problem formulations whatsoever.
For example, what can be said about the minimizers if just the cost function formulation is changed?
What is the ``most meaningful'' way to reformulate the cost and the inequalities in the presence of certain inequalities?
To find an acceptable trade-off between conservatism and computational complexity is a topic worth pursuing.

The presented results hold for transmission networks under \dc power flow assumptions.
It is natural to ask whether the method can be extended to the more general case of \ac power flow.
In fact, \cite{Muehlpfordt16b,Engelmann18} studied \ac-\opf under uncertainty using polynomial chaos and showed that \pce provides policies of higher order than affine policies that satisfy the \ac power flow equations \emph{numerically}; however no rigorous proof is provided.
It remains an open research question what kinds of policies generally satisfy \ac power flow.

Finally, in applications the data-driven computation of the \pce of the uncertain demand/feed-in may itself be challenging and subject to uncertainty about the uncertainty.
Thus, distributionally robust \pce formulations are of interest in the future.

\section*{Acknowledgment}
All authors would like to express their gratitude towards Lutz Gr\"oll from the Institute for Automation and Applied Informatics at Karlsruhe Institute of Technology.
The numerous insightful discussions with him helped to improve the manuscript.

\section*{References}
\bibliographystyle{elsarticle-num} 
\def\bibfont{\footnotesize}
\bibliography{PCE_DCOPF_lit}
\normalsize


\end{document}